\documentclass[11pt,twoside]{article}
\usepackage{amsmath,amssymb,theorem}
\usepackage{graphicx}
\usepackage{amssymb}

\usepackage[dvips]{epsfig}

\newcommand{\donothing}[1]{{}}

\numberwithin{equation}{section}
\newtheorem{theorem}{Theorem}
\newtheorem{corollary}{Corollary}
\newtheorem{lemma}{Lemma}

\newtheorem{remark}{Remark}
\newenvironment{proof}{{\bf Proof}.\ }{ \hfill $\square$}
\newcommand{\R}{\mathbb{R}}

\newcommand{\Z}{\mathbb{Z}}
\newcommand{\N}{\mathbb{N}}

\begin{document}

\title{Enstrophy growth in the viscous Burgers equation}

\author{Dmitry Pelinovsky \\
{\small Department of Mathematics, McMaster
University, Hamilton, Ontario, Canada, L8S 4K1}}

\date{\today}
\maketitle

\begin{abstract}
We study bounds on the enstrophy growth for solutions of the viscous Burgers equation
on the unit circle. Using the variational formulation of Lu and Doering, we prove
that the maximizer of the enstrophy's rate of change is sharp in the limit of large enstrophy up to
a numerical constant but does not saturate the Poincar\'e inequality for mean-zero $1$-periodic
functions. Using the dynamical system methods,
we give an asymptotic representation of the maximizer in the limit of large
enstrophy as a viscous shock on the background of a linear rarefactive wave.
This asymptotic construction is used to prove that a larger growth of enstrophy
can be achieved  when the initial data to the viscous Burgers equation saturates the Poincar\'e inequality
up to a numerical constant.

An exact self-similar solution of the Burgers equation is constructed to describe formation
of a metastable viscous shock on the background of a linear rarefactive wave.
When we consider the Burgers equation on an infinite line
subject to the nonzero (shock-type) boundary conditions,
we prove that the maximum enstrophy achieved in the time evolution is scaled as $\mathcal{E}^{3/2}$,
where $\mathcal{E}$ is the large initial enstrophy, whereas the time needed for reaching the maximal
enstrophy is scaled as $\mathcal{E}^{-1/2} \log(\mathcal{E})$. Similar but slower rates are proved
on the unit circle.
\end{abstract}

\section{Introduction}

We consider the initial-value problem for the one-dimensional viscous Burgers equation,
\begin{equation}
\label{Burgers}
\left\{ \begin{array}{l} u_t + 2u u_x = u_{xx}, \;\;\; x \in \mathbb{T}, \; t \in \mathbb{R}_+,\\
u|_{t = 0} = u_0, \phantom{texttext} x \in \mathbb{T},\end{array} \right.
\end{equation}
where $\mathbb{T} = \left[-\frac{1}{2},\frac{1}{2}\right]$ is the unit circle
equipped with the periodic boundary conditions for the real-valued function $u$.
Local well-posedness of the initial-value problem (\ref{Burgers}) holds for
$u_0 \in H^s_{\rm per}(\mathbb{T})$ with $s > -\frac{1}{2}$ \cite{Dix}.
The Burgers equation is used as a toy model in the context of a bigger
problem of how to control existence and regularity of solutions
of the three-dimensional Navier--Stokes equations \cite{Burgers,Lorenz}.
Recent applications of the Burgers equation to the theory of turbulence
can be found in \cite{Ohkitani,Shivamoggi}.

Lu and Doering \cite{LuLu} considered the question of optimal bounds on the enstrophy growth.
The enstrophy for the Burgers equation (\ref{Burgers}) is defined by
\begin{equation}
\label{enstrophy}
E(u) = \frac{1}{2} \int_{\mathbb{T}} u_x^2 dx.
\end{equation}
Integration by parts for a strong local solution
of the Burgers equation (\ref{Burgers}) in $H^3_{\rm per}(\mathbb{T})$ yields
\begin{eqnarray}
\label{rate-of-change-E}
\frac{d E(u)}{d t} = \int_{\mathbb{T}} u_x (u_{xxx} - 2 u u_{xx} - 2 u_x^2 ) dx
= -\int_{\mathbb{T}} ( u_{xx}^2 + u_x^3 ) dx \equiv R(u),
\end{eqnarray}
where $R(u)$ is the rate of change of $E(u)$.

If $u \in {\cal C}_{\rm per}^1(\mathbb{T})$, then there is $\xi \in \mathbb{T}$ such that $u_x(\xi) = 0$.
Using the elementary bound,
$$
u_x^2(x) = \left( \int_{\xi}^x - \int_x^{1 + \xi} \right) u_x u_{xx} dx \quad \Rightarrow \quad
\| u_x \|_{L^{\infty}} \leq \| u_x \|^{1/2}_{L^2} \| u_{xx} \|^{1/2}_{L^2},
$$
and the Young inequality for $a,b \in \R_+$,
$$
a b \leq \frac{a^p}{p \epsilon^p} + \frac{\epsilon^q b^q}{q}, \quad
\frac{1}{p} + \frac{1}{q} = 1, \quad \epsilon > 0,
$$
the rate of change $R(u)$ in (\ref{rate-of-change-E}) can be estimated by
\begin{equation}
\label{bound-on-R}
|R(u)| \leq - \| u_{xx} \|_{L^2}^2 + \| u_x \|_{L^2}^{5/2} \| u_{xx} \|_{L^2}^{1/2} \leq
\frac{3}{4^{4/3}} \| u_x \|_{L^2}^{10/3} \equiv \frac{3}{2} E^{5/3}(u),
\end{equation}
provided that $p = \frac{4}{3}$, $q = 4$, and $\epsilon = \sqrt{2}$.

In the framework of the Burgers equation (\ref{Burgers}),
Lu and Doering \cite{LuLu} showed that the bound
$|R(u)| \leq C E^{5/3}(u)$ on the enstrophy growth is sharp
in the limit of large enstrophy, up to a choice of the numerical constant $C > 0$.
To prove the claim, they considered the maximization problem,
\begin{equation}
\label{max-problem}
\max_{u \in H^2_{\rm per}(\mathbb{T})} R(u) \quad \mbox{\rm subject to} \quad E(u) = {\cal E},
\end{equation}
for a given value of ${\cal E} > 0$. An analytical solution
of the constrained maximization problem (\ref{max-problem}) was
studied in the asymptotic limit of large ${\cal E}$ by using Jacobi's elliptic functions.
We note that the bound (\ref{bound-on-R}) is achieved instantaneously in time
and it may not hold for solutions of the viscous Burgers equation (\ref{Burgers})
for a finite time interval.

Ayala and Protas \cite{Diego} reiterated the same question on the validity of
bound (\ref{bound-on-R}) integrated over a finite time interval. The
energy balance equation for the Burgers equation (\ref{Burgers}) is given by
\begin{equation}
\label{energy}
K(u) = \frac{1}{2} \int_{\mathbb{T}} u^2 dx \quad \Rightarrow \quad
\frac{d K(u)}{d t} = \int_{\mathbb{T}} u (u_{xx} - 2 u u_x) dx = - 2 E(u).
\end{equation}
If bound (\ref{bound-on-R}) is sharp on the time interval $[0,T]$
for some $T > 0$, then integration of the enstrophy equation (\ref{rate-of-change-E}) implies
\begin{equation}
\label{nonlocal-bound}
E^{1/3}(u(T)) - E^{1/3}(u_0) \leq \frac{1}{2} \int_0^T E(u(t)) dt
= \frac{1}{4} \left[ K(u_0) - K(u(T)) \right].
\end{equation}

The Burgers equation (\ref{Burgers}) maps the set of periodic functions with zero mean
to itself. Using the Poincar\'e inequality for periodic functions with zero mean,
\begin{equation}
\label{Poincare-inequality}
K(u_0) \leq \frac{1}{4 \pi^2} E(u_0),
\end{equation}
and neglecting $K(u(T))$ in (\ref{nonlocal-bound}), we can obtain
\begin{equation}
\label{bound-on-E}
E(u(T)) \leq \left( E^{1/3}(u_0) + \frac{1}{16 \pi^2} E(u_0) \right)^3.
\end{equation}
Note that this bound together with the monotonicity of
$K(u(t)) \leq K(u_0)$ implies global well-posedness of the initial-value
problem (\ref{Burgers}) for any $u_0 \in H^1_{\rm per}(\mathbb{T})$.

Using the extended maximization problem for the global
solution of the Burgers equation (\ref{Burgers})
in $H^1_{\rm per}(\mathbb{T})$,
\begin{equation}
\label{max-nonlocal-problem}
\max_{u_0 \in H^1_{\rm per}(\mathbb{T})} E(u(T)) \quad \mbox{\rm subject to} \quad E(u_0) = {\cal E},
\end{equation}
Ayala and Protas \cite{Diego} showed numerically that the integral bound (\ref{bound-on-E}) is not sharp
even in the limit of large ${\cal E}$.

We shall use the notation $A = \mathcal{O}({\cal E}^p)$ as $\mathcal{E} \to \infty$
if there are constants $C_{\pm}$ such that $0 \leq C_- < C_+ < \infty$
and $C_- \mathcal{E}^p \leq A \leq C_+ \mathcal{E}^p$.
Let $T_*$ be the value of $T$, where $\max_{u_0 \in H^1_{\rm per}(\mathbb{T})} E(u(T))$ is
maximal over $T \in \R_+$. The main claims in \cite{Diego} are reproduced
in Table I.

\vspace{0.2cm}

\begin{center}
\begin{tabular}{|l|l|l|l|l|}
  \hline
Initial condition & Time $T_*$ &
Enstrophy at $T_*$ & Energy $K$ at $T_*$ \\
  \hline
A maximizer of (\ref{max-problem}) & ${\cal O}({\cal E}^{-0.6})$ & ${\cal O}({\cal E}^{1.0})$
& ${\cal O}({\cal E}^{0.7})$ \\
  \hline
A maximizer of (\ref{max-nonlocal-problem}) & ${\cal O}({\cal E}^{-0.5})$ & ${\cal O}({\cal E}^{1.5})$
& ${\cal O}({\cal E}^{1.0})$ \\
   \hline
\end{tabular}
\end{center}

{\bf Table I:} Enstrophy growth in the Burgers equation
from the numerical results in \cite{Diego}.

\vspace{0.2cm}

The first line in Table I shows that
the instantaneous maximizer of the problem (\ref{max-problem}) does not saturate
the Poincar\'e inequality (\ref{Poincare-inequality}) and does not lead to large growth
of the enstrophy. On the other hand, the second line in Table I shows
that the bound (\ref{bound-on-E}) is not sharp. The bound (\ref{nonlocal-bound})
could be sharp if $K(u_0) - K(u(T_*)) = {\cal O}({\cal E}^{1/2})$
but the numerical work in \cite{Diego} reported large deviations in
numerical approximations of this quantity,
\begin{equation}\label{suspicious-bound}
K(u_0) - K(u(T_*)) = \mathcal{O}({\cal E}^{0.68 \pm 0.25}),
\end{equation}
which may indicate that the underlying relation may have a logarithmic (or other) correction.

In this paper, we shall study further properties of the analytical solution
of the constrained maximization problem (\ref{max-problem}). We shall use this solution
and its generalizations (see Section 2)
as an initial condition for the Burgers equation (\ref{Burgers}).
In particular, we shall address rigorously the numerical results of \cite{Diego}.
Our main results are summarized in Table II.

\vspace{0.2cm}

\begin{center}
\begin{tabular}{|l|l|l|l|l|}
  \hline
Initial condition & Time $T_*$ &
Enstrophy at $T_*$ & Energy $K$ at $T_*$ \\
  \hline
(\ref{initial-data}) and (\ref{instant-function}) & ${\cal O}({\cal E}^{-2/3} \log(\mathcal{E}))$ &
${\cal O}({\cal E})$ & ${\cal O}({\cal E}^{2/3})$ \\
  \hline
(\ref{initial-data}) and (\ref{particular-function})  &
${\cal O}({\cal E}^{-1/2} \log^{1/2}(\mathcal{E}))$ &
${\cal O}({\cal E}^{3/2} \log^{-3/2}(\mathcal{E}))$ &
${\cal O}({\cal E}\log^{-1}(\mathcal{E}))$ \\
   \hline
\end{tabular}
\end{center}

{\bf Table II:} Enstrophy growth in the Burgers equation from our analytical results.

\vspace{0.2cm}

The analytical results in Table II justify partially the results of numerical approximations
in Table I. We conjecture that the optimal rate is achieved with
\begin{equation}
\label{future-rates}
T_* = {\cal O}({\cal E}^{-1/2}), \quad E(u(T_*))
= {\cal O}({\cal E}^{3/2}), \quad K(u(T_*)) = {\cal O}({\cal E}),
\end{equation}
but we have no proof of this rate at the present time, perhaps, due to technical
limitations of our method (see Remarks \ref{remark-obstacle-1} and \ref{remark-obstacle-2}).
Similarly, we cannot derive an analytical analogue of
the numerical result (\ref{suspicious-bound}) and hence, the sharpness of the nonlocal
bound (\ref{nonlocal-bound}) remains an open question for further studies.

From a technical point of view, using the dynamical system methods,
we prove that the maximizer of the constrained
maximization problem (\ref{max-problem}) does not saturate the Poincar\'e
inequality (\ref{Poincare-inequality}). In the limit of large enstrophy ${\cal E}$,
this maximizer resembles a viscous shock on
the background of a linear rarefactive wave.
If this maximizer is taken as the initial data to the viscous Burgers equation
(\ref{Burgers}), it does not give the largest change of enstrophy,
compared to the case when the initial data saturates the Poincar\'e inequality (\ref{Poincare-inequality}).
On the other hand, if the shock's width is used as an independent parameter
relative to the background intensity of the linear rarefactive wave, the initial data
can saturate the Poincar\'e inequality (\ref{Poincare-inequality})
for large values of $\mathcal{E}$, up to a numerical
constant, and achieve a faster growth of enstrophy in the time evolution
of the viscous Burgers equation.

We note that our construction of the viscous shocks on the background of a linear rarefactive wave
is similar to the diffusive $N$-waves that appear at the intermediate stages
of dynamics of arbitrary initial data in the Burgers equation over an infinite line \cite{Kim}.
However, these metastable states correspond to Gaussian fundamental solutions of the
heat equation in self-similar variables \cite{Wayne1,Wayne2}, whereas
our solutions are obtained on a circle of large but finite period after a scaling transformation.
Our results still rely on the analysis of the Burgers equation over an infinite
line subject to the non-zero (shock-type) boundary conditions, where
viscous shocks are known to be asymptotically stable \cite{Satt,Goodman,Szepessy,Zumbrun}.

The technique of this paper does not use much of the Cole--Hopf transformation \cite{Cole,Hopf},
which is well known to reduce the viscous Burgers equation to the linear heat equation.
The transformation is only used in Sections 5 and 6 to reduce technicalities in
the convergence analysis for dynamics of viscous shocks in bounded and unbounded domains.
At the same time, one can think about other techniques to prove the
conjecture (\ref{future-rates}), which explore the Cole--Hopf transformation in
more details. In particular, this transformation shows that solutions of the viscous Burgers
equation can be studied with the Laplace method for the heat equation. The Laplace method is typically
used to recover solutions of the inviscid Burgers equations from
solutions of the viscous Burgers equation in the limit of vanishing viscosity
(see, e.g., \cite[Chapter 2]{Whitham} or \cite[Chapter 3]{Miller}).
Note that the limit of vanishing viscosity corresponds to the limit of large
enstrophy in the context of our work. There is a promising way to prove the conjecture (\ref{future-rates})
using the Laplace method after the shock of the corresponding inviscid Burgers equation
is formed and this task will be the subject of an independent work \cite{JonGood}.

The paper is organized as follows.  Section 2 presents main results.
Solutions of the constrained maximization problem (\ref{max-problem}) are
characterized in Section 3. Self-similar solutions of the Burgers
equation on the unit circle are considered in Section 4.
Section 5 presents analysis of the Burgers equation on an infinite line.
Evolution of a viscous shock on the background of a linear rarefactive wave
is studied in Section 6. Proofs of the main results for
two different initial data in Table II are given in Sections 7 and 8.


\section{Main results}

We shall first reexamine the solution of the constrained maximization problem (\ref{max-problem}).
Unlike the work of Lu and Doering \cite{LuLu}, we avoid the use of special
functions (the Jacobi elliptic functions) but use dynamical system techniques
to study the limit of large enstrophy ${\cal E}$. As a result,
we obtain the following theorem. Here $\check{H}^2_{\rm per}(\mathbb{T})$ denotes
the restriction of $H^2_{\rm per}(\mathbb{T})$ to odd functions
and $A = \mathcal{O}_{L^{\infty}}(B)$ as $B \to \infty$ indicates that
$\| A \|_{L^{\infty}} = \mathcal{O}(B)$.

\vspace{1cm}

\begin{theorem}
For sufficiently large $\mathcal{E}$,
there exists a unique solution $u_* \in \check{H}^2_{\rm per}(\mathbb{T})$
of the constrained maximization problem (\ref{max-problem}) with $u_*'(0) < 0$
satisfying
\begin{equation}
\label{profile-u-star}
u_*(x) = 4 k (2x - \tanh(kx)) + \mathcal{O}_{L^{\infty}}(k^2 e^{-k}), \quad \mbox{\rm as} \quad k \to \infty,
\end{equation}
where $k$ determines the leading order expansions,
\begin{eqnarray}
\label{expansion-K-1}
K(u_*) & = & \frac{8}{3} k^2 + \mathcal{O}(k), \\
\label{expansion-K-2}
E(u_*) & = & \frac{32}{3} k^3 + \mathcal{O}(k^2), \\
\label{expansion-K-3}
R(u_*) & = & \frac{256}{5} k^5 + \mathcal{O}(k^4).
\end{eqnarray}
\label{theorem-local-bound}
\end{theorem}

\begin{corollary}
\label{corollary-local-bound}
When $k$ is expressed from (\ref{expansion-K-2}) in terms of $\mathcal{E} = E(u_*)$, we obtain
\begin{equation}
\label{asympt-expansions}
\left. \begin{array}{l}
K(u_*) = \frac{1}{6^{1/3}} \mathcal{E}^{2/3} + \mathcal{O}(\mathcal{E}^{1/3}), \\
R(u_*) = \frac{3^{5/3}}{5 \cdot 2^{1/3}} \mathcal{E}^{5/3} + \mathcal{O}(\mathcal{E}^{4/3}),
\end{array} \right\} \quad \mbox{\rm as} \quad \mathcal{E} \to \infty.
\end{equation}
\end{corollary}

\begin{remark}
\label{remark-local-bound}
Corollary \ref{theorem-local-bound} improves the earlier claims in \cite{Diego} and in \cite{LuLu} based on numerical
and asymptotic computations, respectively. It shows that the Poincare inequality (\ref{Poincare-inequality})
is not saturated by the solution of the constrained maximization problem (\ref{max-problem}), whereas
the bound $|R(u)| \leq C E^{5/3}(u)$ is sharp up to a choice of the numerical constant $C > 0$ with
$$
C = \frac{3^{5/3}}{5 \cdot 2^{1/3}} < \frac{1}{2}.
$$
\end{remark}

We shall now consider the time evolution of the Cauchy problem (\ref{Burgers}) with the initial data
\begin{equation}
\label{initial-data}
u_0(x) = 4k (2x - f(x)), \quad x \in \mathbb{T},
\end{equation}
where $k > 0$ is a free parameter and $f : \mathbb{T} \to \R$ is a fixed function satisfying
\begin{equation}
\label{initial-function}
f \in \mathcal{C}^1(\mathbb{T}) : \quad f(-x) = -f(x), \;\; f\left(\frac{1}{2}\right) = 1.
\end{equation}

The maximizer of Theorem \ref{theorem-local-bound} is represented by (\ref{profile-u-star}).
Neglecting the exponentially small terms as $k \to \infty$,
this maximizer can be written in the form (\ref{initial-data}) with
\begin{equation}
\label{instant-function}
f(x) = \frac{\tanh(kx)}{\tanh(k/2)}.
\end{equation}
We say that the initial data (\ref{initial-data}) with (\ref{instant-function})
represents a shock on the background of a linear rarefactive wave, where the width of the shock
is inverse proportional to large parameter $k$.

The maximizer of Theorem \ref{theorem-local-bound} does not saturate
the Poincar\'e inequality (\ref{Poincare-inequality}) in the limit
$k \to \infty$ (Remark \ref{remark-local-bound}). To allow more flexibility,
we can take the initial data (\ref{initial-data}) in the form,
\begin{equation}
\label{particular-function}
f(x) = \frac{\tanh(l x)}{\tanh(l/2)},
\end{equation}
where parameter $l > 0$ may be independent of the large parameter $k$.
Figure \ref{figure-Data} shows both functions (\ref{instant-function})
and (\ref{particular-function}) in the initial data (\ref{initial-data})
by dashed and solid lines, respectively. If $k = 20$ and $l = 5$,
the shock in (\ref{particular-function}) is much smoother than the shock in
(\ref{instant-function}).

\begin{figure}
\begin{center}
  \includegraphics[width=0.65\textwidth]{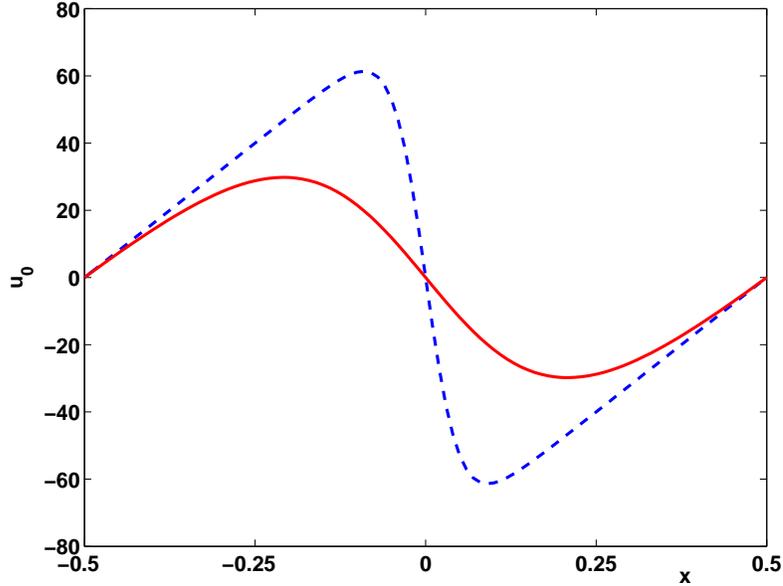}
\end{center}
\caption{Initial data (\ref{initial-data}) with (\ref{instant-function}) (dotted)
or (\ref{particular-function}) (solid) for $k = 20$ and $l = 5$.}
\label{figure-Data}
\end{figure}

If $u_0$ is given by (\ref{initial-data}) with (\ref{particular-function}), then
\begin{equation}
\label{asympt-initial-data}
K(u_0) = k^2 \tilde{K}(l), \quad E(u_0) = k^2 \tilde{E}(l),
\end{equation}
where
\begin{eqnarray*}
\tilde{K}(l) & = & \frac{32}{3} + \frac{8}{\sinh^2(l/2)} - \frac{16 \cosh(l/2)}{l \sinh(l/2)}
\left[ 1 + \frac{l}{2} + 2 \log(1 + e^{-l}) - \frac{2}{l} \int_0^l \log(1 + e^{-z}) dz \right],\\
\tilde{E}(l) & = & \frac{32 l (\cosh(l) + 2)}{3 \sinh(l)} - 32.
\end{eqnarray*}

If $l = k$, then these expansions yield (\ref{expansion-K-1}) and
(\ref{expansion-K-2}) up to the error terms. In this case, if $\mathcal{E} = E(u_0)$ is
fixed, then
\begin{equation}
\label{case-I}
k = \mathcal{O}(\mathcal{E}^{1/3}), \quad K(u_0) = \mathcal{O}(\mathcal{E}^{2/3}), \quad \mbox{\rm as}
\quad \mathcal{E} \to \infty.
\end{equation}

The function $F(l) = \tilde{K}(l)/\tilde{E}(l)$ is plotted on Figure \ref{figure-Poincare} (left).
We can see that there is a maximum of the function $F$ at $l = l_0 \approx 3.0$,
where the maximum is
at
$$
F(l_0) \approx 0.025297 < \frac{1}{4 \pi^2} \approx 0.025330.
$$
If $l$ is fixed independently of $k$ and if $\mathcal{E} = E(u_0)$ is fixed, then
\begin{equation}
\label{case-II}
k = \mathcal{O}(\mathcal{E}^{1/2}), \quad K(u_0) = \mathcal{O}(\mathcal{E}), \quad \mbox{\rm as}
\quad \mathcal{E} \to \infty.
\end{equation}
This shows that the initial data (\ref{initial-data}) with (\ref{particular-function})
saturates the Poincar\'e inequality (\ref{Poincare-inequality}) in the limit $k \to \infty$
up to a numerical constant. Note that the value of the constant prefactor
$F(l_0)$ for $l_0 \approx 3.0$ is $99.9\%$ of the Poincare constant, compared to
the numerical computations in \cite{Diego}, where this prefactor was
found from solutions of the extended maximization problem (\ref{max-nonlocal-problem})
to be $97\%$ of the Poincare constant.

To apply our method, we shall consider a slow (logarithmic) growth of
the parameter $l$ in the limit $k \to \infty$, which yields rates slower
than rates (\ref{case-II}). In particular, we shall use the following elementary result.

\begin{lemma}
Fix $\Delta > 0$ and let $l := (1 + \Delta) \log(k)$. Then, we have
\begin{equation}
\label{case-III}
k = \mathcal{O}(\mathcal{E}^{1/2} \log^{-1/2}(\mathcal{E})), \quad
K(u_0) = \mathcal{O}(\mathcal{E} \log^{-1}(\mathcal{E})), \quad \mbox{\rm as}
\quad \mathcal{E} \to \infty.
\end{equation}
\label{lemma-exp-asymptotics}
\end{lemma}

\begin{proof}
As $k,l \to \infty$, the leading-order expression for $K(u_0)$ and $E(u_0)$ are given by
\begin{equation}
\label{expansions-K-E}
K(u_0) = \frac{8}{3} k^2 + \mathcal{O}\left(\frac{k^2}{l}\right), \quad
E(u_0) = \frac{32}{3} k^2 l + \mathcal{O}(k^2).
\end{equation}
With the choice of $l := (1 + \Delta) \log(k)$, we are to solve $E(u_0) = \mathcal{E}$,
which is equivalent to the implicit equation,
\begin{equation}
\label{implicit-equation}
z = x \log(x) + \mathcal{O}(x), \quad z := \frac{3 \mathcal{E}}{16 (1 + \Delta)}, \quad x := k^2.
\end{equation}
We look at the asymptotic limit $z \to \infty$. Setting
$$
x := \frac{z u}{\log(z)},
$$
we rewrite the implicit equation (\ref{implicit-equation}) as the root finding problem,
\begin{equation}
\label{root-finding}
f(u) := u \left( 1 - \frac{\log(\log(z))}{\log(z)} + \frac{\log(u)}{\log(z)} +
\mathcal{O}\left(\frac{1}{\log(z)}\right) \right) - 1 = 0
\end{equation}
We note that $f(1) = \mathcal{O}\left(\frac{\log(\log(z))}{\log(z)}\right) \to 0$ as $z \to \infty$
and $f'(1) = 1 + \mathcal{O}\left(\frac{\log(\log(z))}{\log(z)}\right) \neq 0$ as $z \to \infty$.
By the implicit function theorem, there is a unique root in the neighborhood of $u = 1$
such that $u = 1 + \mathcal{O}\left(\frac{\log(\log(z))}{\log(z)}\right)$ as $z \to \infty$.
The assertion (\ref{case-III}) holds by (\ref{expansions-K-E}) and (\ref{implicit-equation}).
\end{proof}

\begin{figure}
\begin{center}
  \includegraphics[width=0.48\textwidth]{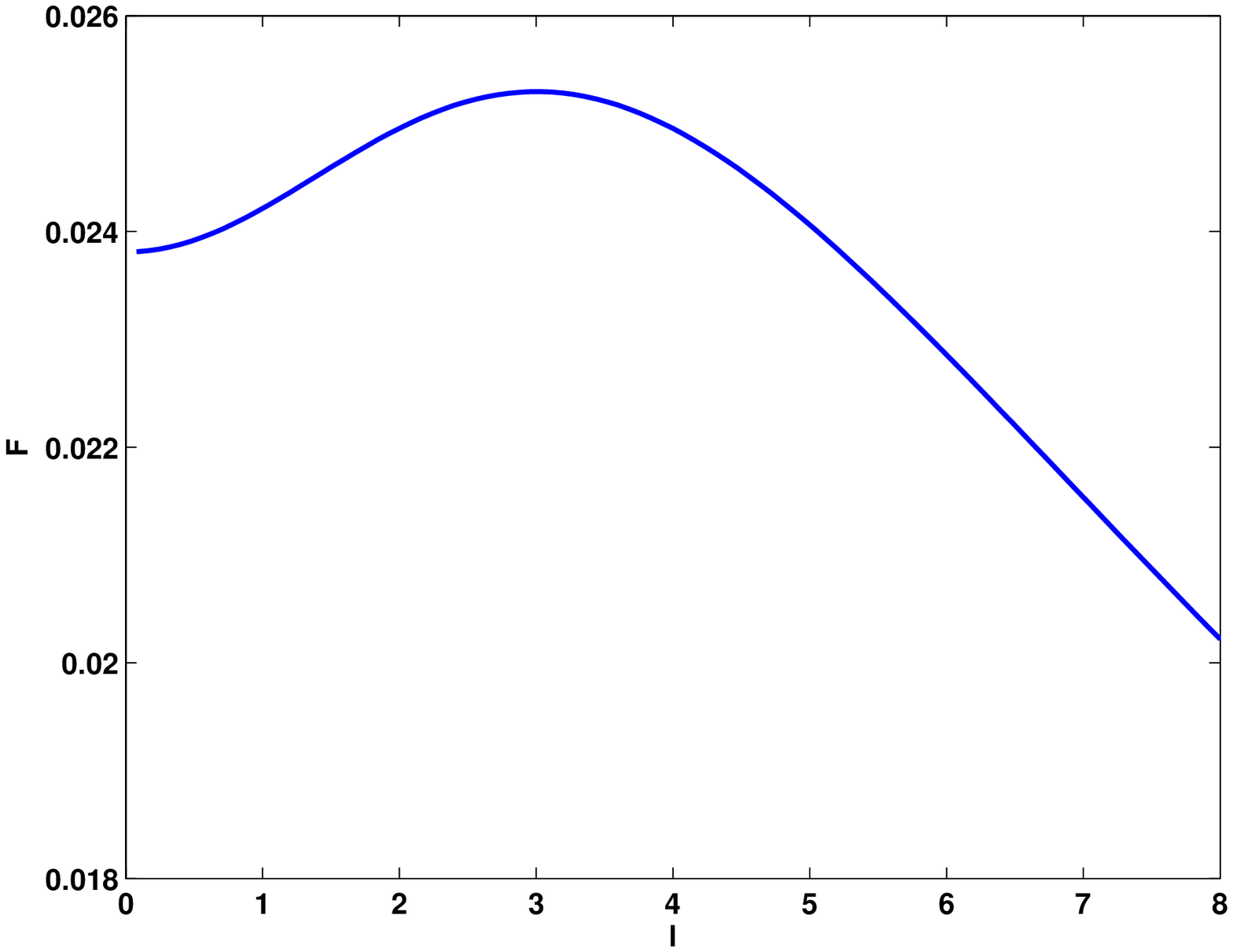}
   \includegraphics[width=0.48\textwidth]{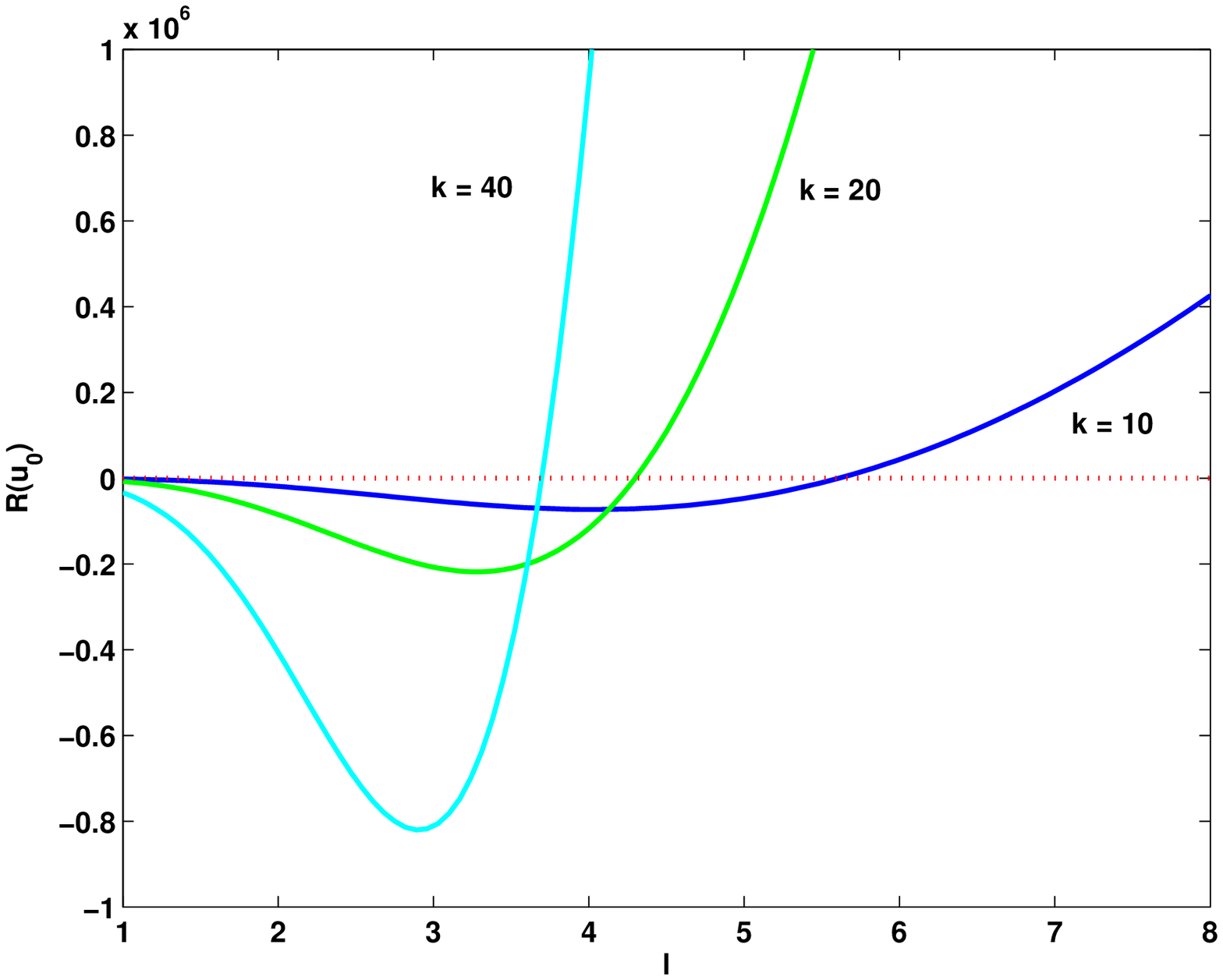}
\end{center}
\caption{Function $F(l) = \tilde{K}(l)/\tilde{E}(l)$ versus $l$ (left)
and function $R(u_0)$ versus $l$ for different values of $k$ (right).}
\label{figure-Poincare}
\end{figure}

If $u_0$ is given by (\ref{initial-data}) and (\ref{particular-function}), then
\begin{eqnarray*}
R(u_0) & = & 64 k^3 \left[ 16 - \frac{12 l}{\tanh^2(l/2)} \int_0^{l/2} {\rm sech}^4(x) dx +
\frac{2 l^2}{\tanh^3(l/2)} \int_0^{l/2} {\rm sech}^6(x) dx \right] \\
& \phantom{t} & \phantom{texttext}
- \frac{128 k^2 l^3}{\tanh^2(l/2)} \int_0^{l/2} \sinh^2(x) {\rm sech}^6(x) dx.
\end{eqnarray*}
The dependence of $R(u_0)$ versus $l$ for different values of $k$ is shown
on Figure \ref{figure-Poincare} (right). Since $R(u_0) > 0$ for large values of $k$
and $l$, the enstrophy $E(u)$ grows initially for $t > 0$.

We shall construct a solution of the Burgers equation (\ref{Burgers})
starting with the initial data (\ref{initial-data})--(\ref{initial-function}).
We prove that this solution displays dynamics consisting of two phases. In the first
phase, a metastable viscous shock is formed from the function $-4 k f(x)$.
In the second phase, a rarefactive wave associated with the linear function
$8 k x$ decays to zero. We compute the growth of enstrophy in two cases:
when $l = \mathcal{O}(k)$ and the scaling law (\ref{case-I}) holds
and when $l = \mathcal{O}(\log(k))$ and the scaling law (\ref{case-III}) holds.
The following theorem gives the main result of this paper.

\begin{theorem}
Consider the initial-value problem (\ref{Burgers}) with
initial data in (\ref{initial-data}) and (\ref{particular-function}).
Let $\mathcal{E} = E(u_0)$ be the initial enstrophy.
There exists $T_* > 0$ such that the enstrophy $E(u)$ achieves its maximum at
$u_* = u(\cdot,T_*)$. If $l = \mathcal{O}(k)$ as $k \to \infty$,
then
\begin{equation}
T_* = \mathcal{O}(\mathcal{E}^{-2/3} \log(\mathcal{E})), \quad
E(u_*) = \mathcal{O}(\mathcal{E}), \quad
K(u_*) = \mathcal{O}(\mathcal{E}^{2/3}), \quad \mbox{\rm as} \quad \mathcal{E} \to \infty,
\label{arg-max}
\end{equation}
whereas if $l = \mathcal{O}(\log(k))$ as $k \to \infty$, then
\begin{equation}
T_* = \mathcal{O}(\mathcal{E}^{-1/2} \log^{1/2}(\mathcal{E})), \quad
E(u_*) = \mathcal{O}(\mathcal{E}^{3/2}\log^{-3/2}(\mathcal{E})), \quad
K(u_*) = \mathcal{O}(\mathcal{E} \log^{-1}(\mathcal{E})),
\label{arg-max-theorem}
\end{equation}
as $\mathcal{E} \to \infty$.
\label{theorem-time}
\end{theorem}

\begin{remark}
Several obstacles appear in our method when we consider the case $l = \mathcal{O}(1)$
and the scaling law (\ref{case-II}). These obstacles come from the behavior of the solution
$u(x,t)$ near the boundaries $x = \pm \frac{1}{2}$ as well as from the constraints on
the inertial time interval $[0,T]$, during which the solution approaches the viscous shock on
the background of a linear rarefactive wave. \label{remark-obstacle-1}
\end{remark}

\section{Proof of Theorem \ref{theorem-local-bound}}

To prove Theorem \ref{theorem-local-bound}, we obtain a convenient analytical representation
of solutions of the constrained maximization problem (\ref{max-problem}).
We set $v = u_x$ and look for critical points $v \in H^1_{\rm per}(\mathbb{T})$ of the functional,
\begin{equation}
\label{functional-J}
J(v) = \int_{\mathbb{T}} \left( v_x^2 + v^3 + \lambda v^2 + \mu v \right) dx,
\end{equation}
where $\lambda, \mu \in \R$ are the Lagrange multipliers associated with
the following constraints,
\begin{equation}
\label{constraints-J}
\frac{1}{2} \int_{\mathbb{T}} v^2(x) dx = {\cal E}, \quad \int_{\mathbb{T}} v(x) dx = 0.
\end{equation}
The latter constraint ensures that $u$ is a periodic function on $\mathbb{T}$.
If $v$ is even in $H^1_{\rm per}(\mathbb{T})$ and has zero mean,
then $u(x) = \int_0^x v(x') dx'$ is odd in $H^2(\mathbb{T})$,
$u(\pm 1) = 0$, and hence, $u \in \check{H}^2_{\rm per}(\mathbb{T})$.

The Euler--Lagrange equations associated with the functional $J(v)$
yield the second-order differential equation,
\begin{equation}
\label{ELequations-J}
2 v''(x) = 3 v^2(x) + 2 \lambda v(x) + \mu, \quad x \in \mathbb{T}.
\end{equation}
Integrating equation (\ref{ELequations-J}) over $\mathbb{T}$ and using
the constraints (\ref{constraints-J}), we find $\mu = - 6 {\cal E}$. Hence we
are dealing with the family of integrable second-order equations,
\begin{equation}
\label{ELequations}
\frac{d^2 v}{d x^2} = \frac{3}{2} v^2 + \lambda v - 3 {\cal E} \quad \Rightarrow \quad
\left( \frac{d v}{d x} \right)^2 = v^3 + \lambda v^2 - 6 {\cal E} v + I,
\end{equation}
where $I$ is an integration constant.

\begin{figure}
\begin{center}
  \includegraphics[width=0.65\textwidth]{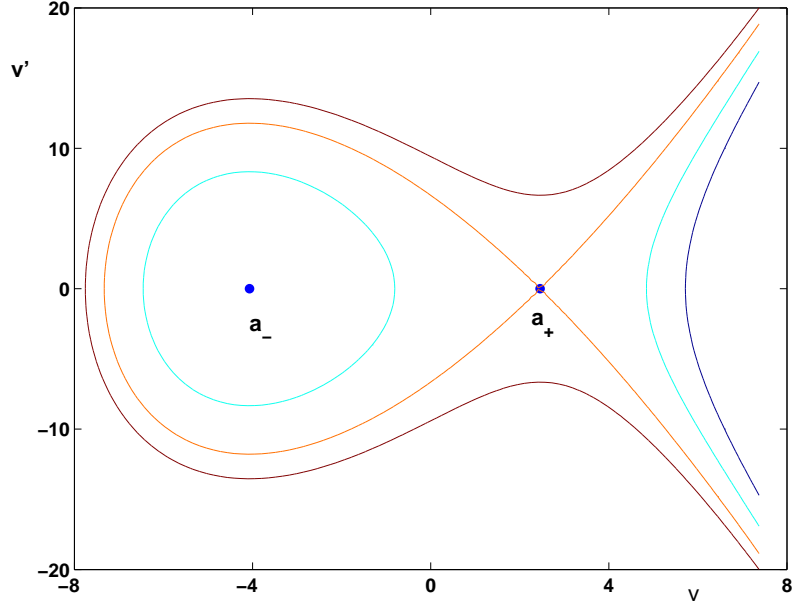}
\end{center}
\caption{Phase portrait for the Euler--Lagrange equation (\ref{ELequations}) with $\mathcal{E} = 5$.}
\label{figure-Phase}
\end{figure}

The phase plane of system (\ref{ELequations}) is given by $(v,v') \in \R^2$.
A typical phase portrait is shown on Figure \ref{figure-Phase}.
There exist two equilibrium points of the second-order equation (\ref{ELequations}), denoted by
$(a_-,0)$ and $(a_+,0)$, where
\begin{equation}
\label{roots-a-pm}
a_{\pm} = \frac{1}{3} \left( - \lambda \pm \sqrt{\lambda^2 + 18 {\cal E}} \right) \quad
\Rightarrow \quad a_- < 0 < a_+.
\end{equation}
Let us define
\begin{equation}
I_{\pm} = -a_{\pm}^3 - \lambda a_{\pm}^2 + 6 {\cal E} a_{\pm} \quad \Rightarrow \quad I_- < I_+.
\end{equation}
The equilibrium $(a_-,0)$ is a center, whereas the equilibrium $(a_+,0)$ is a saddle point.
For $I = I_+$, there exists a homoclinic orbit connecting the stable and unstable manifolds
of the saddle point $(a_+,0)$. This orbit can be found
analytically,
\begin{equation}
\label{relation-k-lambda}
v(x) = a_+ - 4 k^2 {\rm sech}^2(kx), \quad k = \frac{1}{2} \sqrt[4]{\lambda^2 + 18 {\cal E}}.
\end{equation}
Inside the separatrix loop, there is a family of $T$-periodic orbits for $I \in (I_-,I_+)$,
such that $T$ is a strictly increasing function of $I$ with
\begin{equation}
T \to \frac{2 \pi}{\omega} \quad \mbox{\rm as} \quad I \to I_-, \quad
\mbox{\rm and} \quad T \to \infty \quad
\mbox{\rm as} \quad I \to I_+,
\end{equation}
where
$$
\omega^2 = - \lambda - 3 a_- = \sqrt{\lambda^2 + 18 {\cal E}} = 4 k^2.
$$

If $k > \pi$, there is a unique $I_0 \in (I_-,I_+)$ such that $T = 1$ at $I = I_0$.
The corresponding $1$-periodic solution $v \in H^1_{\rm per}(\mathbb{T})$ of equation
(\ref{ELequations-J}) is a critical point of $J(v)$. Integrating (\ref{ELequations}) for
$I = I_0$ further, we obtain a $1$-periodic solution $v(x-x_0)$, where $x_0 \in \mathbb{T}$
is chosen uniquely from the constraints $v'(0) = 0$ and $v(0) < a_- < 0$ (which yield
an odd $u$ with $u'(0) < 0$). In this way, a critical point of $J(v)$ is obtained and
the parameter $\lambda$ needs to be defined by the constraint $E(u) = \mathcal{E}$.

To satisfy the constraint $E(u) = {\cal E}$ and to justify the asymptotic expansion
(\ref{profile-u-star}), we can use the representation,
\begin{equation}
\label{representation-J}
v(x) = a_+ - 4k^2 y(\xi),\quad \xi = k x,
\end{equation}
where $y(\xi)$ is a $k$-periodic solution of the second-order equation,
\begin{equation}
\label{cnoidal-wave}
y'' - 4 y + 6 y^2  = 0.
\end{equation}
The $k$-periodic solution of equation (\ref{cnoidal-wave}), which is called a {\em cnoidal wave},
is equivalent to a $k$-periodic sequence of homoclinic solutions, which are called
{\em solitary waves} \cite{Boyd} (see also Chapter 3 in \cite{Pava}).
This representation uses the theory of Jacobi's elliptic functions.
We can obtain an equivalent approximation result by using methods of the dynamical system theory \cite{Scheel}.

\begin{lemma}
There are $k_0 > 0$ and $C > 0$ such that for all $k \in (k_0,\infty)$,
the $k$-periodic solution $y(\xi)$ of the second-order equation (\ref{cnoidal-wave}) is close
to the solitary wave $y_{\infty}(\xi) = {\rm sech}^2(\xi)$ with the error bound,
\begin{equation}
\label{approx-theorem}
\sup_{\xi \in [-k/2,k/2]} \left| y(\xi) - y_{\infty}(\xi) \right| \leq C e^{-k}.
\end{equation}
\end{lemma}

\begin{proof}
We write a $k$-periodic sequence of the solitary waves as
$$
y_0(\xi) = \sum_{n \in \Z} {\rm sech}^2(\xi + k n)
$$
and decompose the solution $y$ of equation (\ref{cnoidal-wave}) as $y(\xi) = y_0(\xi) + Y(\xi)$.
After straightforward computations, $Y$ satisfies
$$
L Y = F(Y),
$$
where $L = -\partial_{\xi}^2 + 4 - 12 {\rm sech}^2(\xi)$ and
$$
F(Y) = 6 Y^2 + 12 \sum_{n \in \Z \backslash \{0\}} {\rm sech}^2(\xi + k n) Y +
6 \sum_{n \in \Z} \sum_{l \in \Z \backslash \{n\}} {\rm sech}^2(\xi + k n) {\rm sech}^2(\xi + k l).
$$

The operator $L : H^2(\R) \to L^2(\R)$ has a one-dimensional kernel spanned by
the odd function $y_{\infty}'(\xi)$.
The rest of the spectrum of $L$ includes an isolated eigenvalue at $-5$ and the continuous
spectrum for $[4,\infty)$. Hence the operator $L$ is invertible in the space of
even functions.

Let $\mathbb{T}_k := \left[-\frac{k}{2},\frac{k}{2}\right]$ and denote the restriction
of $H^2_{\rm per}(\mathbb{T}_k)$ to even functions by $\hat{H}^2_{\rm per}(\mathbb{T}_k)$.
Let $\hat{L} : \hat{H}^2_{\rm per}(\mathbb{T}_k) \to \hat{L}^2_{\rm per}(\mathbb{T}_k)$
be a restriction of $L$. Then, $\hat{L}$ is invertible if $k$ is sufficiently large,
and existence and uniqueness of small solutions $Y \in H^2_{\rm per}(\mathbb{T}_k)$
of the fixed-point problem $Y = \hat{L}^{-1} F(Y)$ can be found by using contraction mapping arguments,
provided that
$$
\left\| \sum_{n \in \Z \backslash \{0\}} {\rm sech}^2(\cdot + k n) \right\|_{L^{\infty}(\mathbb{T}_k)} \quad \mbox{\rm and}
\quad \left\| \sum_{n \in \Z} \sum_{l \in \Z \backslash \{n\}} {\rm sech}^2(\cdot + k n)
{\rm sech}^2(\cdot + k l) \right\|_{L^2(\mathbb{T}_k)}
$$
converge to zero as $k \to \infty$. We check this by explicit computations. There are
$k_0 > 0$ and $C_1, C_2 > 0$ such that for all $k \in (k_0,\infty)$, we have
\begin{equation}
\label{bound-1}
\left\| \sum_{n \in \Z \backslash \{0\}} {\rm sech}^2(\cdot + k n) \right\|_{L^{\infty}(\mathbb{T}_k)} \leq
\sum_{n \in \N} \frac{8 e^{-(2n-1)k}}{(1 + e^{-(2 n-1)k})^2} \leq C_1 e^{-k}
\end{equation}
and
\begin{eqnarray}
\nonumber
& \phantom{t} & \left\| \sum_{n \in \Z} \sum_{l \in \Z \backslash \{n\}} {\rm sech}^2(\cdot + k n)
{\rm sech}^2(\cdot + k l) \right\|_{L^2(\mathbb{T}_k)} \\ \nonumber
& \phantom{t} &  \phantom{text} \leq \left\| {\rm sech}^2(\cdot) \right\|_{L^2(\mathbb{T}_k)}
\left\| \sum_{l \in \Z \backslash \{ 0 \}} {\rm sech}^2(\cdot + k l)
\right\|_{L^{\infty}(\mathbb{T}_k)}  \\\nonumber
& \phantom{t} & \phantom{texttexttexttexttexttext}
+ 2 \sum_{n \in \N} \left\| {\rm sech}^2(\cdot + k n) \right\|_{L^{\infty}(\mathbb{T}_k)}
\left\| \sum_{l \in \Z \backslash \{n\}} {\rm sech}^2(\cdot + k l) \right\|_{L^2(\mathbb{T}_k)} \\ \nonumber
& \phantom{t} & \phantom{text} \leq 3 \left\| {\rm sech}^2(\cdot) \right\|_{L^2(\mathbb{T}_k)}
\left\| \sum_{l \in \Z \backslash \{ 0 \}} {\rm sech}^2(\cdot + k l)
\right\|_{L^{\infty}(\mathbb{T}_k)} \\\nonumber
& \phantom{t} & \phantom{texttexttexttexttexttext}
+ 4 k^{1/2} \sum_{n \in \N} \sum_{l \in \N} \left\| {\rm sech}^2(\cdot + k n) \right\|_{L^{\infty}(\mathbb{T}_k)}
\left\| {\rm sech}^2(\cdot + k l) \right\|_{L^{\infty}(\mathbb{T}_k)} \\
& \phantom{t} & \phantom{text} \leq 2 \sqrt{3} C_1 e^{-k} + 4 C_1^2 k^{1/2} e^{-2k} \leq C_2 e^{-k}.
\label{bound-2}
\end{eqnarray}
Existence and uniqueness of $Y$ with the error bound (\ref{approx-theorem}) follow
from bounds (\ref{bound-1}) and (\ref{bound-2}), the contraction mapping arguments
for $Y = \hat{L}^{-1} F(Y)$ in $\hat{H}^2_{\rm per}(\mathbb{T}_k)$, and the Sobolev
embedding of $H^2_{\rm per}(\mathbb{T}_k)$ to $L^{\infty}_{\rm per}(\mathbb{T}_k)$.
\end{proof}

To define uniquely $\lambda$ in terms of ${\cal E}$, we recall the constraints (\ref{constraints-J}),
which are equivalent to the scalar equation,
\begin{equation}
\label{constraint-y}
a_+ = 4 k \int_{-k/2}^{k/2} y(\xi) d\xi = 8 k (1 + {\cal O}(k e^{-k})).
\end{equation}
Because $a_+$ is related to $k$ by the exact expression (\ref{roots-a-pm}),
constraint (\ref{constraint-y}) yields a relationship between $k$ and ${\cal E}$ given by
\begin{equation}
\label{root-find}
a_+ = \frac{6 \mathcal{E}}{\lambda + 4 k^2} \quad \Rightarrow \quad
32 k^3 \left( 1 + \mathcal{O}(k e^{-k}) \right) =  \frac{6 \mathcal{E}}{1 + \sqrt{1 - \frac{9\mathcal{E}}{8 k^4}}}.
\end{equation}
Hence we obtain
\begin{equation}
\label{approx-solution}
k = \left( \frac{3}{32} {\cal E} \right)^{1/3} + 1 + \mathcal{O}(\mathcal{E}^{-1/3})
\quad \mbox{\rm as} \quad {\cal E} \to \infty,
\end{equation}
which yields as $\mathcal{E} \to \infty$,
\begin{eqnarray*}
\lambda & = & 4 k^2 \sqrt{1 - \frac{9\mathcal{E}}{8 k^4}} = \left( \frac{3}{4} {\cal E} \right)^{2/3} + \mathcal{O}(\mathcal{E}^{1/3})
\end{eqnarray*}
and
\begin{eqnarray*}
a_+ & = & 8 k (1 + {\cal O}(k e^{-k})) = \left( 48 {\cal E} \right)^{1/3} + \mathcal{O}(1).
\end{eqnarray*}
It follows from (\ref{approx-solution}) that
$$
\mathcal{E} = \frac{32}{3} k^3 - 32 k^2 + \mathcal{O}(k), \quad \mbox{\rm as} \quad k \to \infty,
$$
which recovers (\ref{expansion-K-2}).

Using constraints (\ref{constraints-J}), the Euler--Lagrange equation (\ref{ELequations-J}), the representation (\ref{representation-J}), and the constraint (\ref{constraint-y}), we obtain
\begin{eqnarray}
\nonumber
R(u) & = & 2 \lambda {\cal E} + \frac{1}{2} \int_{\mathbb{T}} v^3(x) d x \\ \label{expression-for-R}
& = & 2 \lambda {\cal E} + a_+^2 (4k^2 - a_+) - 32 k^5 \int_{-k/2}^{k/2} y^3(\xi) d \xi.
\end{eqnarray}
Approximations (\ref{approx-theorem}), (\ref{root-find}), and (\ref{approx-solution})
in (\ref{expression-for-R}) yield
$$
R(u) = \frac{256}{5} k^5 + {\cal O}(k^4) \quad \mbox{\rm as} \quad k \to \infty,
$$
which recovers (\ref{expansion-K-3}).

To obtain (\ref{profile-u-star}) and (\ref{expansion-K-1}),
we integrate the solution $v(x)$ and write
\begin{equation}
\label{representation-U}
u(x) = \int_0^x v(x') dx' = a_+ x - 4k z(k x),
\end{equation}
where
$$
z(\xi) = \int_0^{\xi} y(\xi') d \xi', \quad \xi = k x.
$$
Therefore, we have
\begin{eqnarray*}
K(u) & = & \frac{1}{24} a_+^2 - 8 k a_+ \int_0^{1/2} x z(kx) dx + 8 k \int_{-k/2}^{k/2} z^2(\xi) d\xi.
\end{eqnarray*}
It follows from (\ref{approx-theorem}) that for all large $k$, there is $C > 0$ such that
$z(\xi)$ is close to $z_{\infty}(\xi) = \tanh(\xi)$ with the error bound,
\begin{equation}
\label{approx-theorem-for-z}
\sup_{\xi \in [-k/2,k/2]} \left| z(\xi) - z_{\infty}(\xi) \right| \leq C k e^{-k},
\end{equation}
which recovers (\ref{profile-u-star}) thanks to the expansion (\ref{constraint-y}).

Integrating by parts, we also find the elementary expansion,
$$
\int_0^{1/2} x \tanh(kx) dx = \frac{1}{8} + \frac{1}{4 k^2} \int_0^{\infty} \log(1 + e^{-z}) dz +
{\cal O}(k^{-1} e^{-k}) \quad \mbox{\rm as} \quad k \to \infty.
$$
Together with (\ref{constraint-y}) and (\ref{approx-solution}), this expansion yields
$$
K(u) = \frac{8}{3} k^2 - 16 k + 32 \int_0^{\infty} \log(1 + e^{-z}) dz + {\cal O}(k^3 e^{-k}))
\quad \mbox{\rm as} \quad k \to \infty,
$$
which recovers (\ref{expansion-K-1}).

To complete the proof of Theorem \ref{theorem-local-bound}, we need to show that
the critical point $u \in \check{H}^2_{\rm per}(\mathbb{T})$ of $J(u_x)$ is a maximizer of $R(u)$.
However, this follows from the uniqueness of the critical point $u_*$, for which $R(u_*) > 0$,
and from the fact that $R(u)$ is bounded from above by the bound (\ref{bound-on-R}). Therefore,
the proof of Theorem \ref{theorem-local-bound} is complete.

\section{Burgers equation on the unit circle}

To develop the proof of Theorem \ref{theorem-time}, we convert the initial-value problem
for the Burgers equation (\ref{Burgers}) with initial data (\ref{initial-data})--(\ref{initial-function})
to a convenient form, which separates the decay of the linear rarefactive wave
and the relative dynamics of a shock on the background of the rarefactive wave.

\begin{lemma}
\label{lemma-solution-Burgers}
Let $u_0$ be given by (\ref{initial-data})--(\ref{initial-function}). Then,
a unique solution $u \in C(\R_+,H^1_{\rm per}(\mathbb{T}))$
of the Burgers equation (\ref{Burgers}) is given by
\begin{equation}
\label{representation-u-1}
u(x,t) = p(t) \left( 2 x - w(\xi(x,t),\tau(t)) \right), \quad x \in \mathbb{T}, \quad t \in \R_+,
\end{equation}
where
\begin{equation}
\label{representation-u-2}
p(t) = \frac{4k}{1 + 16 k t}, \quad \xi(x,t) = \frac{4 k x}{1 + 16 k t}, \quad
\tau(t) = \frac{16 k^2 t}{1 + 16 k t},
\end{equation}
and $w(\xi,\tau)$ is a unique odd solution of the Burgers equation,
\begin{equation}
\label{Burgers-rescaled}
\left\{ \begin{array}{l} w_{\tau} = 2 w w_{\xi} + w_{\xi \xi}, \quad
\; |\xi| < 2 (k - \tau),
\quad \tau \in (0,k), \\ w|_{\tau = 0} = f(\xi/4k), \quad
\;\; |\xi| \leq 2k,\end{array} \right.
\end{equation}
subject to the boundary conditions $w = \pm 1$ at $\xi = \pm 2(k - \tau)$.
\end{lemma}

\begin{proof}
Although the proof can be constructed by a direct substitution, we will give
all intermediate details. Recall that the initial-value problem
(\ref{Burgers}) has a unique global solution $u \in C(\R_+,H^1_{\rm per}(\mathbb{T}))$
if the initial data $u_0$ satisfies (\ref{initial-data})--(\ref{initial-function}).
Odd solutions in $x$ are preserved in the time evolution of the Burgers equation
(\ref{Burgers}) and the Sobolev embedding of $H^1_{\rm per}(\mathbb{T})$ to
$C_{\rm per}(\mathbb{T})$ implies that the boundary conditions $u(\pm \frac{1}{2},t) = 0$
are preserved for all $t > 0$.

Let us look for the exact solution of the Burgers equation (\ref{Burgers})
in the separable form,
$$
u(x,t) = p(t) \left(2 x - U(x,q(t)) \right),
$$
where $p(t)$, $q(t)$, and $U(x,q)$ are new variables. If we choose $\dot{p} = -4 p^2$ and
$\dot{q} = p$ starting with $p(0) = p_0$ and $q(0) = 0$,
then $U(x,q)$ satisfies the initial-value problem,
\begin{equation}
\label{Burgers-intermediate}
\left\{ \begin{array}{l} U_q + 4 x U_x = 2 U U_x + \frac{1}{p(t)} U_{xx}, \quad
x \in \mathbb{T}, \quad q > 0,\\
U|_{t = 0} = f(x), \qquad \phantom{textttext} \qquad
x \in \mathbb{T}, \end{array} \right.
\end{equation}
subject to the boundary conditions $U = \pm 1$ at $x = \pm \frac{1}{2}$.
In addition, $U$ is odd in $x$ for any $q > 0$.

We find from the differential equations $\dot{p} = -4 p^2$ and
$\dot{q} = p$ that
$$
p(t) = \frac{p_0}{1 + 4 p_0 t}, \quad q(t) = \frac{1}{4} \log(1 + 4 p_0 t).
$$
Solving equation $U_q + 4 x U_x = 0$ along the characteristics, we define
$$
\frac{dx}{dq} = 4x \quad \Rightarrow \quad x = C e^{4 q},
$$
where $C$ is an integration constant. Noting that $p(t) = p_0 e^{-4q}$, we define
$(\xi,\tau)$ by
$$
\frac{\partial \tau}{\partial q} = p_0 e^{-4q}, \quad
\frac{\partial \xi}{\partial q} = - 4 \xi, \quad \frac{\partial \tau}{\partial x} = p_0 e^{-4q}.
$$
Integrating these equations, we obtain the substitution,
$$
U(x,q) = w(\xi,\tau), \quad \xi = p_0 x e^{-4q} = \frac{p_0 x}{1 + 4 p_0 t}, \quad
\tau = \frac{1}{4} p_0 (1 - e^{-4q}) = \frac{p_0^2 t}{1 + 4 p_0 t},
$$
which transforms (\ref{Burgers-intermediate}) to (\ref{Burgers-rescaled}).
Odd functions $U$ in $x$ become odd functions $w$ in $\xi$
and the boundary conditions $U = \pm 1$ at $x = \pm \frac{1}{2}$ become
the boundary conditions $w = \pm 1$ at $\xi = \pm \frac{1}{2}(p_0 - 4 \tau)$.
Setting $p_0 = 4k$ yields (\ref{representation-u-1}), (\ref{representation-u-2}),
and (\ref{Burgers-rescaled}).
\end{proof}

\begin{remark}
\label{remark-obstacle-2}
The scaling law (\ref{future-rates}) formally follows from the similarity transformation
(\ref{representation-u-1}). If there exists an inertial range $C_- \leq k t \leq C_+$
for some $k$-independent constants $0 < C_- < C_+ < \infty$, where the $H^1$-norm
of $w$ in $\xi$ is $k$-independent, then in this range, $p(t) = \mathcal{O}(k)$,
$K(u) = \mathcal{O}(k^2)$, $E(u) = \mathcal{O}(k^3)$, where $k = \mathcal{O}(\mathcal{E}^{1/2})$
as $\mathcal{E} = E(u_0) \to \infty$. To prove this claim rigorously, we study
a convergence of solutions of the rescaled Burgers equation (\ref{Burgers-rescaled})
starting with the initial condition $w_0(\xi) = f(\xi/4k)$
to the $k$-independent viscous shock $w_{\infty}(\xi) = \tanh(\xi)$
in a bounded but large domain $|\xi| \leq 2 (k-\tau)$ for $\tau = \mathcal{O}(k)$.
For a control of error terms, we have to specify further restriction
on the initial condition $f(\xi/4k)$, which result in weaker statements
(\ref{arg-max}) and (\ref{arg-max-theorem}) of Theorem \ref{theorem-time}.
\end{remark}

The Burgers equation $w_{\tau} = 2 w w_{\xi} + w_{\xi \xi}$ admits the viscous shock,
\begin{equation}
\label{exact-front-attractor}
w_{\infty}(\xi) = \tanh(\xi).
\end{equation}
In the initial-value problem,
\begin{equation}
\label{Burgers-attractor}
\left\{ \begin{array}{l} w_{\tau} = 2 w w_{\xi} + w_{\xi \xi}, \quad
\xi \in \R, \quad \tau \in \R_+, \\ w|_{\tau = 0} = g(\xi/a), \quad
\;\;\;\; \xi \in \R,\end{array} \right.
\end{equation}
where $a > 0$ is a parameter, the viscous shock is an asymptotically stable
attractor in the space of odd functions $g$ with fast (exponential) decay to $\pm 1$ as
$\xi \to \pm \infty$ \cite{Goodman}. To be able to deal with
the dynamics of viscous shocks in the initial-value problem (\ref{Burgers-rescaled})
on a bounded domain, we shall first clarify the dynamics of viscous shocks
in the initial-value problem (\ref{Burgers-attractor}) on the infinite line.

\section{Burgers equation on the infinite line}

Let us rewrite the initial-value problem (\ref{Burgers-attractor})
for the Burgers equation on the infinite line by using the original (unscaled)
variables,
\begin{equation}
\label{Burgers-line}
\left\{ \begin{array}{l} u_t + 2u u_x = u_{xx}, \quad x \in \R, \;\; t \in \R_+, \\
u|_{t = 0} = u_0, \quad \quad \quad \;\; x \in \R, \end{array} \right.
\end{equation}
We impose the nonzero (shock-type) boundary conditions,
\begin{equation}
\label{boundary-conditions}
\lim_{x \to \pm \infty} u(x,t) = \mp U_{\infty}, \quad t \in \R_+,
\end{equation}
for some $U_{\infty} > 0$.

To control the enstrophy on the infinite line, we define
\begin{equation}
\label{enstrophy-line}
E(u) = \frac{1}{2} \int_{\R} u_x^2 dx, \quad R(u) = -\int_{\R} ( u_{xx}^2 + u_x^3 ) dx, \quad \frac{d E(u)}{d t} = R(u).
\end{equation}
The bound $|R(u)| \leq \frac{3}{2} E^{5/3}(u)$ on the enstrophy growth is derived similarly
to (\ref{bound-on-R}). In the case of the infinite line, the maximizer of $R(u)$ at fixed $E(u)$ is not decaying
at infinity. On the other hand, the result of Theorem \ref{theorem-local-bound} becomes now explicit.

\begin{lemma}
The maximization problem,
\begin{equation}
\label{max-problem-line}
\max_{u_x \in H^1(\R)} R(u) \quad \mbox{\rm subject to} \quad E(u) = {\cal E},
\end{equation}
admits a unique odd solution
\begin{equation}
\label{front}
u_*(x) = - 4k \tanh(kx),
\end{equation}
where $k$ is defined implicitly by ${\cal E}$,
\begin{equation}
\label{E-front}
{\cal E} = E(u_*) = \frac{32}{3} k^3,
\end{equation}
and
\begin{equation}
\label{R-front}
R(u_*) = \frac{256}{5} k^5 = \frac{3^{5/3}}{5 \cdot 2^{1/3}} {\cal E}^{5/3}.
\end{equation}
\label{lemma-R-infty}
\end{lemma}

\begin{proof}
The constrained maximization problem (\ref{max-problem-line}) for $v = u_x \in H^1(\R)$
yields the functional,
$$
J(v) = \int_{\R} \left( v_x^2 + v^3 + \lambda v^2 \right) dx,
$$
where $\lambda \in \R$ is the Lagrange multiplier. The Euler--Lagrange equations give
$$
2 v''(x) = 3 v^2(x) + 2 \lambda v(x), \quad x \in \R
$$
for which the only solution $v \in H^1(\R)$ is the soliton,
$$
v(x) = -4 k^2 {\rm sech}^2(kx), \quad \lambda = 4 k^2,
$$
where $k > 0$ is arbitrary. Integrating $v(x)$ with respect to $x$, we obtain (\ref{front}).
Integrating $v^2$, $v^3$, and $v_x^2$ over $\R$, we obtain (\ref{E-front}) and (\ref{R-front}).
\end{proof}

\begin{remark}
The Burgers equation (\ref{Burgers-line}) subject to the boundary conditions (\ref{boundary-conditions})
with $U_{\infty} = 4k$ admits the viscous shock solution,
\begin{equation}
\label{exact-front}
u_{\infty}(x) = - 4k \tanh(4k x), \quad k > 0,
\end{equation}
which yields $R(u_{\infty}) = 0$ and $E(u_{\infty}) = 4 E(u_*)$.
\end{remark}

\begin{remark}
Using the self-similar variables,
\begin{equation}
\label{variables-hear}
u(x,t) = -4k w(\xi,\tau), \quad \xi = 4 k x, \quad \tau = 16 k^2 t,
\end{equation}
the Burgers equation (\ref{Burgers-line}) can be written in the form,
\begin{equation}
\label{Burgers-line-variable-w}
\left\{ \begin{array}{l} w_{\tau} = 2 w w_{\xi} + w_{\xi \xi}, \quad \xi \in \R, \;\; \tau \in \R_+, \\
w|_{\tau = 0} = w_0, \quad \quad \quad \;\; \xi \in \R, \end{array} \right.
\end{equation}
where $u_0(x) = -4 k w_0(\xi)$. If the boundary conditions (\ref{boundary-conditions})
are imposed with $U_{\infty} = 4k$, then $w$ satisfies the boundary conditions
$\lim_{\xi \to \pm \infty} w(\xi,\tau) = \pm 1$.
\end{remark}

\subsection{Initial condition with $l = k$}

Let us consider the time evolution of the Burgers equations (\ref{Burgers-line}) and
(\ref{Burgers-line-variable-w}) starting with the initial condition,
\begin{equation}
\label{initial-data-inf}
u_0(x) = -4 k \tanh(kx), \quad \Rightarrow \quad w_0(\xi) = \tanh(\xi/4),
\end{equation}
which is a local maximizer (\ref{front}) in Lemma \ref{lemma-R-infty}.
Note that the initial-value problem (\ref{Burgers-line-variable-w}) with
initial data (\ref{initial-data-inf}) is independent of parameter $k > 0$.
We shall prove the following.

\begin{lemma}
For any integer $m \geq 0$, there is a constant $C_m > 0$ such that
a unique solution of the initial-value problem (\ref{Burgers-line-variable-w})
with initial data (\ref{initial-data-inf}) satisfies
\begin{equation}
\label{asymptotic-solution}
\sup_{\xi \in \R} \left| e^{|\xi|/2} \partial_{\xi}^m \left( w(\xi,\tau) - \tanh(\xi) \right)
\right| \leq C_m e^{-3 \tau/4}, \quad \tau \in \R_+.
\end{equation}
\label{lemma-instant-solution}
\end{lemma}

\begin{proof}
Using the Cole--Hopf transformation \cite{Cole,Hopf}, the Burgers equation (\ref{Burgers-line-variable-w})
with initial data (\ref{initial-data-inf}) admits the exact solution,
\begin{equation}
\label{Cole-Hopf}
w(\xi,\tau) = \frac{\psi_{\xi}(\xi,\tau)}{\psi(\xi,\tau)},
\end{equation}
where $\psi(\xi,\tau) > 0$ is a solution of the heat equation $\psi_{\tau} = \psi_{\xi \xi}$
on the real line $\R$ with the initial condition $\psi(\xi,0) = \cosh^4(\xi/4)$.
This exact solution exists in the explicit form,
\begin{equation}
\label{exact-solution-psi}
\psi(\xi,\tau) = \frac{1}{8} \left[ 3 + 4 \cosh(\xi/2) e^{\tau/4} + \cosh(\xi) e^{\tau} \right],
\quad \xi \in \R, \;\; \tau \in \R_+.
\end{equation}

As $\tau \to \infty$, the last term in (\ref{exact-solution-psi}) dominates and the solution
converges in $L^{\infty}$ norm to the viscous shock $w_{\infty}(\xi) = \tanh(\xi)$. To
prove this convergence, we rewrite (\ref{exact-solution-psi}) in the form,
\begin{equation}
\label{representatioin-1}
\psi(\xi,\tau) = \frac{1}{8} e^{\tau} \cosh(\xi) \left[
1 + 4 \cosh(\xi/2) {\rm sech}(\xi) e^{-3 \tau/4} +  3 {\rm sech}(\xi) e^{-\tau} \right].
\end{equation}
This representation and the Cole--Hopf transformation (\ref{Cole-Hopf})
yield the compact expression,
\begin{equation}
\label{compact-form}
w(\xi,\tau) = \tanh(\xi) + \tilde{w}(\xi,\tau),
\end{equation}
where
\begin{equation}
\label{approx-solution-w-2}
\tilde{w} = e^{-3 \tau/4} {\rm sech}(\xi)
\frac{2 \sinh(\xi/2) - 4 \cosh(\xi/2) \tanh(\xi) - 3 \tanh(\xi) e^{-\tau/4}}{1 +
4 \cosh(\xi/2) {\rm sech}(\xi) e^{- 3 \tau/4} + 3 {\rm sech}(\xi) e^{-\tau}}.
\end{equation}
The bound (\ref{asymptotic-solution}) follows from (\ref{approx-solution-w-2})
thanks to the exponential decay in $\xi$ and $\tau$.
\end{proof}

\begin{figure}
\begin{center}
  \includegraphics[width=0.5\textwidth]{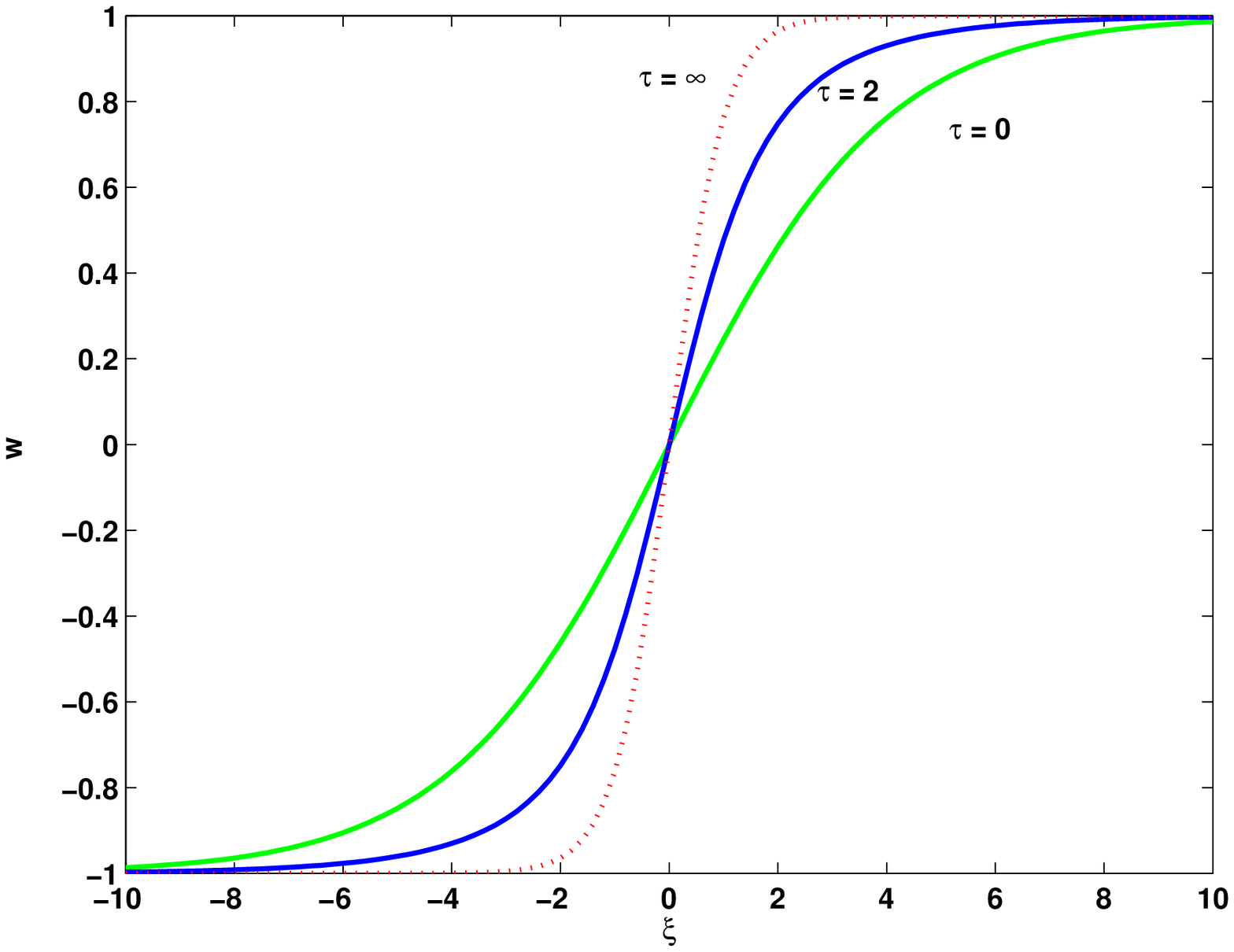}\\
  \includegraphics[width=0.5\textwidth]{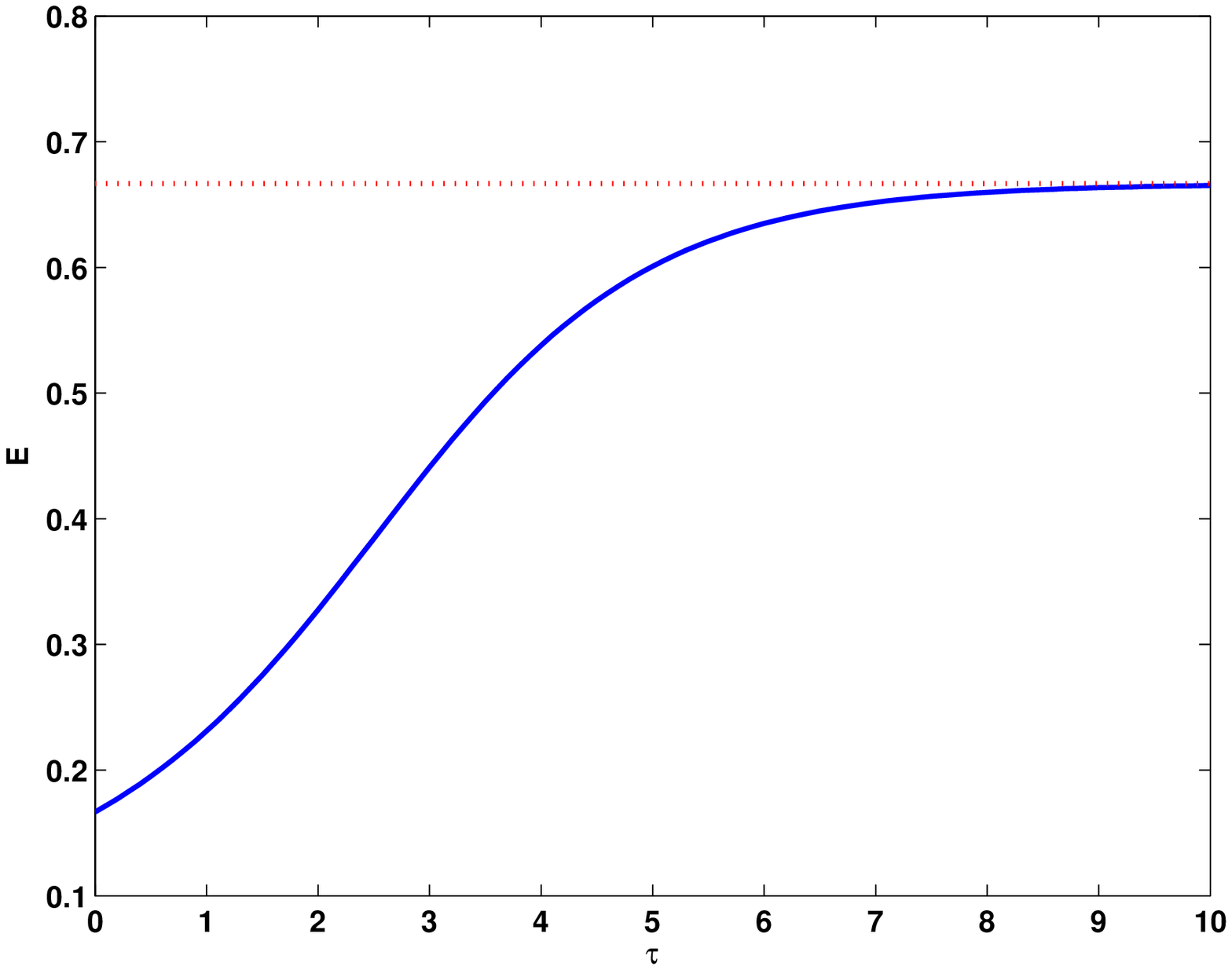}\\
  \includegraphics[width=0.5\textwidth]{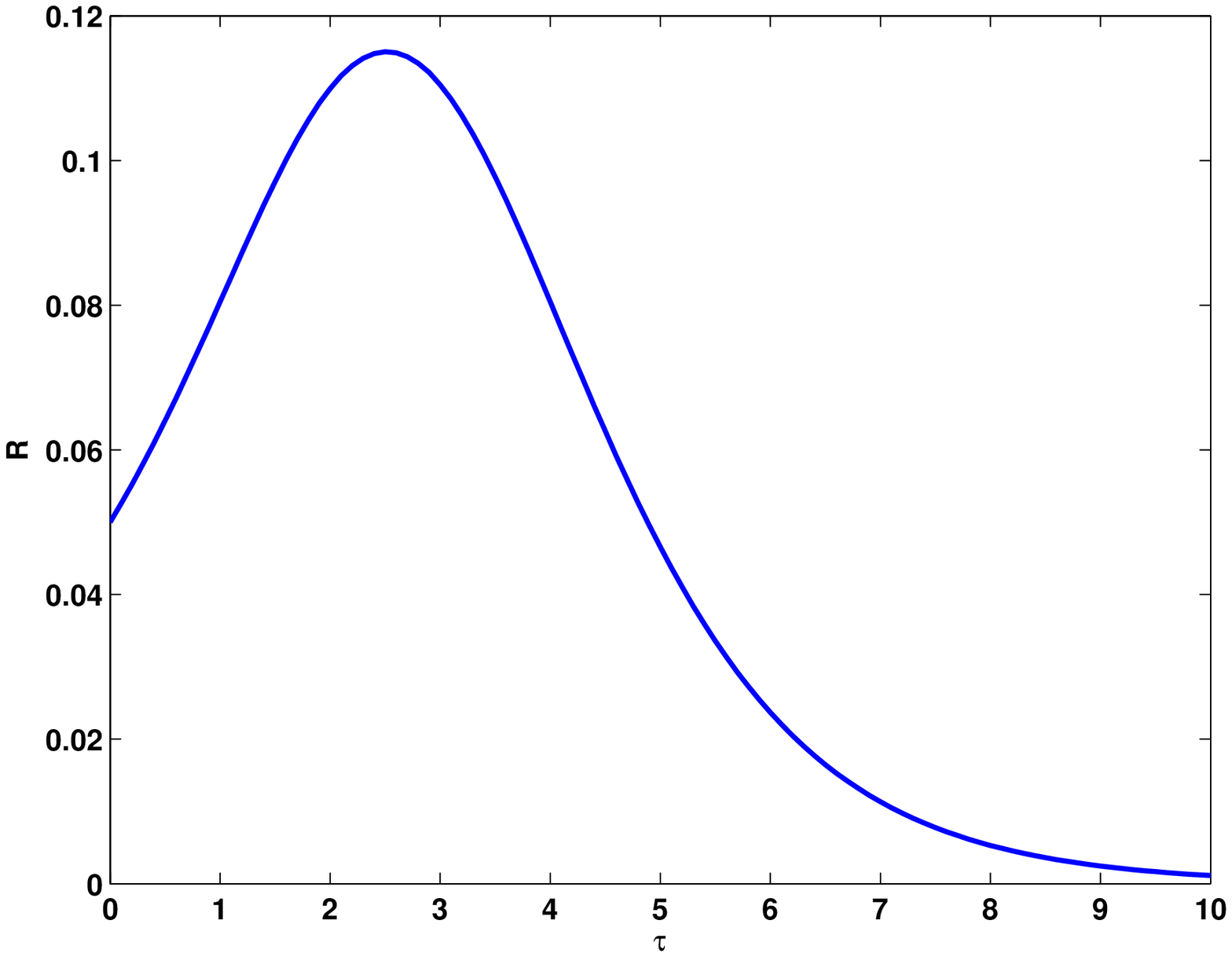}
\end{center}
\caption{Top: The exact solution (\ref{compact-form})--(\ref{approx-solution-w-2})
for different values of $\tau$ and in the limit $\tau \to \infty$.
Middle: Enstrophy $E$ versus time $\tau$. The dotted
curve shows the value of $E$ at the viscous shock $w_{\infty}(\xi) = \tanh(\xi)$.
Bottom: Enstrophy's rate of change $R$ versus time $\tau$.}
\label{figure-2}
\end{figure}

\begin{corollary}
Fix $\delta > 0$. There exist $K > 0$ and $C > 0$ such that
for all $k \geq K$, we have
\begin{equation}
\label{bound-asymptotic-solution}
\sup_{x \in \R} \left| u(x,t) - u_{\infty}(x) \right| \leq \frac{C}{k^{\delta}}, \quad
\mbox{\rm for all} \;\; t \geq T_* := \frac{(1 + \delta) \log(k)}{12 k^2}.
\end{equation}
\label{corollary-1}
\end{corollary}

\begin{proof}
Bound (\ref{bound-asymptotic-solution}) follows from (\ref{asymptotic-solution})
with $m = 0$ after the transformation (\ref{variables-hear}).
\end{proof}

\begin{remark}
We actually have shown convergence of $u$ to $u_{\infty}$ in $H^s$ norm
for any $s \geq 0$, which follows from the exponential decay in $\xi$ and $\tau$.
\end{remark}

Figure \ref{figure-2} shows the exact solution $w(\xi,\tau)$ for different values of $\tau$
together with the rescaled enstrophy $E = \int_0^{\infty} w_{\xi}^2 d\xi$
and its rate of change $R = 2 \int_0^{\infty} (w_{\xi}^3 - w_{\xi \xi}^2) d\xi$ versus $\tau$.
The relevant integrals are approximated by using the MATLAB \verb"quad" function.
The enstrophy $E$ is a monotonically increasing function
to the value $E_{\infty} = \frac{2}{3}$, whereas
the rate of change $R$ is initially increasing and then decreasing
to $R_{\infty} = 0$. Note that by Lemma \ref{lemma-instant-solution},
we have exponentially rapid convergence $\| w(\cdot,\tau) - w_{\infty} \|_{H^1} \to 0$ as
$\tau \to \infty$.

If ${\cal E} = E(u_0) = \frac{32}{3} k^3$,
then bound (\ref{bound-asymptotic-solution}) shows that
\begin{equation}
\label{growth-1}
E(u_{\infty}) = \mathcal{O}(\mathcal{E}) \quad \mbox{\rm and} \quad
T_* = \mathcal{O}(\mathcal{E}^{-2/3} \log(\mathcal{E})) \quad \mbox{\rm as} \quad \mathcal{E} \to \infty.
\end{equation}
The asymptotic result (\ref{growth-1})
is included in the scaling law (\ref{arg-max}) of Theorem \ref{theorem-time}.
We have not achieved the maximal increase of the enstrophy with the
initial data (\ref{initial-data-inf}) given by the instantaneous maximizer of Lemma \ref{lemma-R-infty}.
To achieve the maximal growth of the enstrophy, we shall modify the initial data $u_0$.

\subsection{Initial condition with arbitrary $l$}

Let us consider a more general initial condition,
\begin{equation}
\label{initial-data-inf-modified}
u_0(x) = -4 k \tanh(l x), \quad \Rightarrow \quad w_0(\xi) = \tanh(\xi/a),
\end{equation}
where $a = \frac{4k}{l}$ is the only parameter of the initial-value problem
(\ref{Burgers-line-variable-w}) with initial data
(\ref{initial-data-inf-modified}). For analysis needed in the proof of
Theorem \ref{theorem-time}, we need convergence of the solution $w(\xi,\tau)$ to
the viscous shock $w_{\infty}(\xi) = \tanh(\xi)$ in the $H^1$-norm
as the time $\tau$ gets large, and the smallness of $w(\xi,\tau) - w_{\infty}(\xi)$ at large values of
$\xi$ for all times $\tau \geq 0$. Both objectives are achieved in the following result.

\begin{lemma}
\label{lemma-Burgers-solution}
Fix $\delta > 0$. There are constants $A > 0$ and $C > 0$ such that
for all $a \geq A$, a unique solution of the initial-value problem
(\ref{Burgers-line-variable-w}) with initial data
(\ref{initial-data-inf-modified}) satisfies
\begin{equation}
\label{main-bound-line-infty-norm}
\sup_{\xi \in \R} \left| w(\xi,\tau) - \tanh(\xi) \right|
\leq \frac{C}{a^{3 \delta + \delta^2} \tau^{1/2}}
\quad \mbox{\rm for all} \quad \tau \geq  \frac{1}{2} (1 + \delta)^2 a \log(a)
\end{equation}
Furthermore, for any integer $m \geq 0$, there are constants $A_m > 0$ and $C_m > 0$, such that
for all $a \geq A_m$, we have
\begin{equation}
\label{main-bound-line-infty-boundary}
\left| \partial_{\xi}^m \left( w(\xi,\tau) - \tanh(\xi) \right) \right| \leq \frac{C_m}{a^{1 + m + \delta}}
\quad \mbox{\rm for all} \quad |\xi| \geq \frac{1}{2}(1 + \delta)^2 a \log(a)
\quad \mbox{\rm and} \quad \tau \geq 0.
\end{equation}
\end{lemma}

\begin{proof}
Using the Cole--Hopf transformation (\ref{Cole-Hopf}),
we rewrite the initial-value problem (\ref{Burgers-line-variable-w})
with initial data (\ref{initial-data-inf-modified}) in the form,
\begin{equation}
\label{Burgers-heat}
\left\{ \begin{array}{l} \psi_{\tau} = \psi_{\xi \xi}, \quad \quad \quad \quad \quad \;\;
\xi \in \R, \quad \tau \in \R_+, \\ \psi|_{\tau = 0} = \cosh^a(\xi/a), \quad \xi \in \R.
\end{array} \right.
\end{equation}
The initial data can be represented in the form,
$$
\psi|_{\tau = 0} = \frac{1}{2^a} e^{|\xi|} \phi_a(\xi), \quad \phi_a(\xi) = (1 + e^{-2|\xi|/a})^a.
$$

The function $\phi_a$ is bounded by its value at $\xi = 0$,
\begin{equation}
\label{bound-phi-a-infty}
1 \leq \phi_a(\xi) \leq 2^a, \quad \xi \in \R,
\end{equation}
but the upper bound diverges quickly as $a \to \infty$. On the other hand,
for large values of $\xi$, the function $\phi_a$ decays to $1$ exponentially rapidly.

Fix $\delta > 0$ and define $\xi_a := \frac{1+\delta}{2} a \log(a)$.
For sufficiently large $a \gg 1$, there is $C > 0$ such that
\begin{equation}
\label{bound-phi-a}
0 \leq \phi_a(\xi) - 1 = e^{a \log(1 + e^{-2|\xi|/a})} - 1 \leq C a e^{-2|\xi|/a}
\leq \frac{C}{a^{\delta}}, \quad |\xi| \geq \xi_a.
\end{equation}
The upper bound in (\ref{bound-phi-a}) converges to zero as $a \to \infty$.

To bound derivatives of $\phi_a$ in $\xi$, we note that
\begin{equation}
\label{phi-derivative}
\phi_a'(\xi) = -2 (1 + e^{-2\xi/a})^{a-1} e^{-2\xi/a}, \quad \xi > 0,
\end{equation}
which implies that
\begin{equation}
\label{bound-phi-derivative}
\sup_{\xi \in \R_+} |\phi_a'(\xi)| = 2^a, \quad
\sup_{\xi \geq \xi_a} |\phi_a'(\xi)| \leq \frac{C}{a^{1+\delta}}.
\end{equation}
Writing (\ref{phi-derivative}) as
$$
\frac{d}{d \xi} \log \phi_a(\xi) = -\frac{2 e^{-2\xi/a}}{1 + e^{-2\xi/a}}, \quad
\Rightarrow \quad \frac{d^2}{d \xi^2} \log \phi_a(\xi) = \frac{1}{a}
{\rm sech}^2\left(\frac{\xi}{a}\right), \quad \xi > 0,
$$
and using (\ref{bound-phi-a-infty}) and (\ref{bound-phi-a}), we
obtain by induction that, for any integer $m \geq 0$, there is $C_m > 0$ such that
\begin{equation}
\label{bound-phi-higher-derivative}
\sup_{\xi \in \R_+} |\partial_{\xi}^m \phi_a(\xi)| \leq C_m 2^a, \quad
\sup_{\xi \geq \xi_a} |\partial_{\xi}^m  \phi_a(\xi)| \leq \frac{C_m}{a^{m+\delta}}.
\end{equation}

The solution of the initial-value problem for the heat equation (\ref{Burgers-heat}) is
obtained explicitly as
\begin{eqnarray*}
\psi(\xi,\tau) & = &  \frac{1}{2^a \sqrt{4 \pi \tau}}
\int_{-\infty}^{\infty} e^{-\frac{(\xi - \eta)^2}{4 \tau} + |\eta|}
\phi_a(\eta) d \eta \\
& = & \frac{1}{2^a \sqrt{\pi}} e^{\tau}
\left( e^{\xi} \int_{-z_+}^{\infty} e^{-z^2}
\phi_a(\xi + 2 \tau + 2 \sqrt{\tau} z) dz +
e^{-\xi} \int_{-\infty}^{z_-} e^{-z^2}
 \phi_a(\xi - 2 \tau + 2 \sqrt{\tau} z) dz \right),
\end{eqnarray*}
where
$$
z_{\pm} = \frac{2\tau \pm \xi}{2\sqrt{\tau}}.
$$
We shall rewrite the solution in the equivalent form,
\begin{eqnarray}
\psi(\xi,\tau) = \frac{1}{2^{a-1}} e^{\tau} \cosh(\xi) \left[ 1 + \Psi(\xi,\tau) \right],
\quad \Psi(\xi,\tau) := \frac{e^{\xi} \psi_+(\xi,\tau)
+ e^{-\xi} \psi_-(\xi,\tau)}{\sqrt{\pi} (e^{\xi} + e^{-\xi})},
\label{exact-sol-hear-better}
\end{eqnarray}
where
\begin{eqnarray}
\label{expression-psi-plus}
\psi_+(\xi,\tau) & := & \int_{-z_+}^{\infty} e^{-z^2}
\phi_a(\xi + 2 \tau + 2 \sqrt{\tau} z) dz - \sqrt{\pi}, \\
\label{expression-psi-minus}
\psi_-(\xi,\tau) & := & \int_{-\infty}^{z_-} e^{-z^2}
 \phi_a(\xi - 2 \tau + 2 \sqrt{\tau} z) dz - \sqrt{\pi}.
\end{eqnarray}
Using the Cole--Hopf transformation (\ref{Cole-Hopf}) and the explicit
representation (\ref{exact-sol-hear-better}), we write the solution
of the Burgers equation (\ref{Burgers-line-variable-w}) in the form,
\begin{equation}
\label{compact-form-new}
w(\xi,\tau) = \tanh(\xi) + \tilde{w}(\xi,\tau), \quad
\tilde{w}(\xi,\tau) = \frac{\partial_{\xi} \Psi(\xi,\tau)}{1 + \Psi(\xi,\tau)},
\end{equation}
where
\begin{eqnarray*}
\partial_{\xi} \Psi(\xi,\tau) = \frac{e^{\xi} \partial_{\xi} \psi_+(\xi,\tau) +
e^{-\xi} \partial_{\xi} \psi_-(\xi,\tau)}{\sqrt{\pi} (e^{\xi} + e^{-\xi})} +
\frac{1}{2 \sqrt{\pi}} {\rm sech}^2(\xi) (\psi_+(\xi,\tau) - \psi_-(\xi,\tau)).
\end{eqnarray*}

We first prove bound (\ref{main-bound-line-infty-norm}).
Because $\psi$ is even in $\xi$, we can consider $\xi \geq 0$
without the loss of generality. To analyze $\psi_+$, we define
$$
z_a := z_+ - \frac{\xi_a}{2 \sqrt{\tau}} =
\frac{\xi-\xi_a+2\tau}{2 \sqrt{\tau}},
$$
so that $z \geq -z_a$ corresponds to $\xi \geq \xi_a$ in the argument of $\phi_a(\xi)$.
For any $\xi \geq 0$ and $\tau \geq \tau_a := (1 + \delta) \xi_a$, we have
$$
z_a \geq \frac{2 \tau_a - \xi_a}{2 \sqrt{\tau_a}} \geq
\frac{1+2\delta}{2 \sqrt{2}} \sqrt{a \log(a)},
$$
which shows that $z_a \to \infty$ as $a \to \infty$. The term $\psi_+(\xi,\tau)$
can be split into the sum of three terms
\begin{eqnarray}
\nonumber
\psi_+(\xi,\tau) & = & \left( \int_{-z_a}^{\infty} e^{-z^2} dz - \sqrt{\pi} \right) +
\int_{-z_a}^{\infty} e^{-z^2} (\phi_a(\xi + 2 \tau + 2 \sqrt{\tau} z) - 1) dz  \\
\label{splitting-psi}
& \phantom{t} & \quad + \int_{-z_+}^{-z_a} e^{-z^2}
\phi_a(\xi + 2 \tau + 2 \sqrt{\tau} z) dz \equiv I + II + III.
\end{eqnarray}
For all $\xi \geq 0$ and $\tau \geq \tau_a$, we obtain
$$
I = \int_{-\infty}^{-z_a} e^{-z^2} dz \leq \frac{1}{2 z_a} e^{-z_a^2}
\leq \frac{\sqrt{2}}{(1+2\delta)\sqrt{a \log(a)}} e^{-\frac{(1+2\delta)^2}{8} a \log(a)}
$$
and
\begin{eqnarray*}
III \leq 2^a \int_{-z_+}^{-z_a} e^{-z^2} dz \leq
\frac{2^a}{2 z_a} e^{-z_a^2} \leq \frac{\sqrt{2}}{(1+2\delta)\sqrt{a \log(a)}}
e^{-\frac{(1+2\delta)^2}{8} a \log(a)+a \log(2)},
\end{eqnarray*}
where the upper bounds are exponentially small as $a \to \infty$ and
the bound (\ref{bound-phi-a-infty}) has been used. Using bound (\ref{bound-phi-a}), we obtain
$$
II \leq C a \int_{-z_a}^{\infty} e^{-z^2 - 2(\xi + 2 \tau + 2 \sqrt{\tau} z)/a} dz
\leq \sqrt{\pi} C a e^{-2(1 + \delta)^2 \log(a)} \leq \frac{C}{a^{1 + 4 \delta + 2\delta^2}}.
$$
Combining these terms and dropping the exponentially small terms,
we infer that for any fixed $\delta > 0$ and for
sufficiently large $a \gg 1$, there is $C > 0$ such that
\begin{equation}
\label{bound-first}
\sup_{\tau \geq \tau_a} \sup_{\xi \geq 0} |\psi_+(\xi,\tau)| \leq
\frac{C}{a^{1 + 4 \delta + 2\delta^2}}.
\end{equation}

To analyze $\psi_-$, we define
$$
\hat{z}_a := z_- - \frac{\xi_a}{2 \sqrt{\tau}} = \frac{-\xi-\xi_a+2\tau}{2 \sqrt{\tau}},
$$
so that $z \leq z_a$ corresponds
to $\xi \leq -\xi_a$ in the argument of $\phi_a(\xi)$.
For any $0 \leq \xi \leq \xi_a$ and $\tau \geq \tau_a$, we have
$$
\hat{z}_a \geq \frac{\tau_a - \xi_a}{\sqrt{\tau_a}} \geq
\frac{\delta}{\sqrt{2}} \sqrt{a \log(a)},
$$
which shows that $\hat{z}_a \to \infty$ as $a \to \infty$.
The term $\psi_-(\xi,\tau)$ can be split into the sum of three terms
\begin{eqnarray}
\nonumber
\psi_-(\xi,\tau) & = & \left( \int_{-\infty}^{\hat{z}_a} e^{-z^2} dz - \sqrt{\pi} \right) +
\int_{-\infty}^{\hat{z}_a} e^{-z^2}
(\phi_a(\xi - 2 \tau + 2 \sqrt{\tau} z) - 1) dz  \\
\label{splitting-psi-minus}
& \phantom{t} & + \int_{\hat{z}_a}^{z_-} e^{-z^2}
\phi_a(\xi - 2 \tau + 2 \sqrt{\tau} z) dz \equiv I + II + III.
\end{eqnarray}
Using computations similar to those for $\psi_+(\xi,\tau)$, we obtain
for any $0 \leq \xi \leq \xi_a$ and $\tau \geq \tau_a$,
$$
I \leq \frac{1}{\delta \sqrt{2 a \log(a)}} e^{-\frac{\delta^2}{2} a \log(a)}, \quad
II \leq \frac{C}{a^{3 \delta + 2\delta^2}}, \quad III \leq
\frac{1}{\delta \sqrt{2 a \log(a)}} e^{-\frac{\delta^2}{2} a \log(a) + a \log(2)}.
$$
Again, for any fixed $\delta > 0$ and for
sufficiently large $a \gg 1$, there is $C > 0$ such that
\begin{equation}
\label{bound-second}
\sup_{\tau \geq \tau_a} \sup_{0 \leq \xi \leq \xi_a} |\psi_-(\xi,\tau)| \leq
\frac{C}{a^{3 \delta + 2\delta^2}}.
\end{equation}

Finally, the exponential factor in front of $\psi_-$ yields the following simple estimate:
\begin{equation}
\label{bound-third}
\sup_{\xi \geq \xi_a} e^{-2 \xi} \psi_-(\xi,\tau) \leq \sqrt{\pi}
(2^a - 1) e^{-2 \xi_a} \leq C e^{-(1+\delta) a \log(a) + a \log(2)}, \quad \tau \geq 0,
\end{equation}
where the upper bound is exponentially small as $a \to \infty$.

Because of the symmetry for $\xi \leq 0$, we infer
from (\ref{bound-first}), (\ref{bound-second}), and (\ref{bound-third}) that
for any fixed $\delta > 0$ and for sufficiently large $a \gg 1$,
there is constant $C > 0$ such that
\begin{equation}
\label{bound-1new-Psi}
\sup_{\tau \geq \tau_a} \sup_{\xi \in \R} |\Psi(\xi,\tau) | \leq
\frac{C}{a^{3 \delta + 2\delta^2}}.
\end{equation}
This result shows that $\Psi$ is small for $\tau \geq \tau_a$
in the denominator of the exact solution (\ref{compact-form-new}).

We now proceed with similar expressions for $\partial_{\xi} \psi_+(\xi,\tau)$
and $\partial_{\xi} \psi_-(\xi,\tau)$. Using the representation
(\ref{splitting-psi}), we find
\begin{equation}
\label{derivative-1}
\partial_{\xi} I = \frac{1}{2 \sqrt{\tau}} e^{-z_a^2},
\end{equation}
\begin{eqnarray}
\nonumber
\partial_{\xi} II & = & \frac{1}{2 \sqrt{\tau}} e^{-z_a^2} (\phi_a(\xi_a) - 1) +
\frac{1}{2 \sqrt{\tau}} \int_{-z_a}^{\infty} e^{-z^2} \partial_z \phi_a(\xi + 2 \tau + 2 \sqrt{\tau} z) dz \\
\label{derivative-2}
& = & \frac{1}{\sqrt{\tau}} \int_{-z_a}^{\infty} z e^{-z^2}
\left( \phi_a(\xi + 2 \tau + 2 \sqrt{\tau} z) - 1\right) dz,
\end{eqnarray}
and
\begin{eqnarray}
\nonumber
\partial_{\xi} III & = & \frac{1}{2\sqrt{\tau}} \left[ e^{-z_+^2} \phi_a(0) - e^{-z_a^2} \phi_a(\xi_a) \right]
+ \frac{1}{2\sqrt{\tau}} \int_{-z_+}^{-z_a} e^{-z^2}
\partial_z \phi_a(\xi + 2 \tau + 2 \sqrt{\tau} z) dz \\
& = & \frac{1}{\sqrt{\tau}} \int_{-z_+}^{-z_a} z e^{-z^2}
\phi_a(\xi + 2 \tau + 2 \sqrt{\tau} z) dz.
\label{derivative-3}
\end{eqnarray}
Using bounds (\ref{bound-phi-a-infty}) and (\ref{bound-phi-a}), we obtain that
for any $\xi \geq 0$ and $\tau \geq \tau_a$
there is $C > 0$ such that
$$
\partial_{\xi} I \leq \frac{C}{\sqrt{\tau}}
e^{-\frac{(1+2\delta)^2}{8} a \log(a)}, \quad
|\partial_{\xi} II| \leq \frac{C}{\sqrt{\tau} a^{1 + 4 \delta + 2\delta^2}}, \quad
|\partial_{\xi} III| \leq \frac{C}{\sqrt{\tau}}
e^{-\frac{(1+2\delta)^2}{8} a \log(a) + a \log(2)},
$$
where in the computations for $\partial_{\xi} II$ we have used
the fact that the function $\sqrt{\tau} e^{-4 \tau/a}$ is
monotonically decreasing for any $\tau >\frac{a}{8}$, whereas
$\tau_a \gg \mathcal{O}(a)$ as $a \to \infty$.

Combining all bounds and dropping the exponentially small factors,
we infer that for any fixed $\delta > 0$ and
sufficiently large $a \gg 1$, there is $C > 0$ such that
\begin{equation}
\label{bound-first-new}
\sup_{\xi \geq 0} |\partial_{\xi} \psi_+(\xi,\tau)| \leq
\frac{C}{\sqrt{\tau} a^{1 + 4 \delta + 2\delta^2}}, \quad \tau \geq \tau_a.
\end{equation}
Similarly, we obtain
\begin{equation}
\label{bound-second-new}
\sup_{0 \leq \xi \leq \xi_a} |\partial_{\xi} \psi_-(\xi,\tau)| \leq
\frac{C}{\sqrt{\tau} a^{3 \delta + 2\delta^2}}, \quad \tau \geq \tau_a
\end{equation}
and
\begin{equation}
\label{bound-third-new}
|\partial_{\xi} \psi_-(\xi,\tau)| \leq
\frac{2^a}{\sqrt{\tau}} \int_{-\infty}^{\sqrt{\tau} - \xi/(2\sqrt{\tau})} |z| e^{-z^2} dz,
\quad \tau \geq 0, \quad \xi \geq \xi_a,
\end{equation}
which yields
\begin{equation}
\label{bound-third-new-a}
\sup_{\xi \geq \xi_a} e^{-2 \xi} |\partial_{\xi} \psi_-(\xi,\tau)| \leq
\frac{C 2^a}{\sqrt{\tau}} e^{-2 \xi_a} \leq
\frac{C}{\sqrt{\tau}} e^{-(1+\delta) a \log(a) + a \log(2)}, \quad \tau \geq \tau_a.
\end{equation}
Because of the symmetry for $\xi \leq 0$, we infer
from (\ref{bound-first-new}), (\ref{bound-second-new}), and (\ref{bound-third-new-a}) that
for any fixed $\delta > 0$ and for sufficiently large $a \gg 1$,
there is constant $C > 0$ such that
\begin{equation}
\label{bound-1new}
\sup_{\xi \in \R} \left| \frac{e^{\xi} \partial_{\xi} \psi_+(\xi,\tau) +
e^{-\xi} \partial_{\xi} \psi_-(\xi,\tau)}{e^{\xi} + e^{-\xi}} \right| \leq
\frac{C}{\sqrt{\tau} a^{3 \delta + 2\delta^2}}, \quad \tau \geq \tau_a.
\end{equation}

We now need to estimate the difference $\psi_+(\xi,\tau) - \psi_-(\xi,\tau)$.
We can write
\begin{eqnarray*}
\psi_+(\xi,\tau) - \psi_-(\xi,\tau) & = & \frac{1}{2\sqrt{\tau}} \int_0^{\infty}
\phi_a(\eta) \left[ e^{-\frac{(\eta - \xi - 2 \tau)^2}{4 \tau}} -
e^{-\frac{(\eta + \xi - 2 \tau)^2}{4 \tau}} \right] d \eta \\
& = & \frac{1}{\sqrt{\tau}} \left( \int_{-2\tau}^{-\tau} + \int_{-\tau}^{\infty} \right)
\phi_a(\eta + 2 \tau) e^{-\frac{\xi^2+\eta^2}{4 \tau}}
\sinh\left(\frac{\xi \eta}{2 \tau}\right) d \eta \equiv I + II.
\end{eqnarray*}
For any $|\xi| \leq \frac{1}{2} \tau$, we have
\begin{eqnarray*}
|I| & = & \frac{1}{\sqrt{\tau}} \int_{\tau}^{2\tau} \phi_a(2 \tau - \eta)
e^{-\frac{\xi^2+\eta^2}{4 \tau}} \sinh\left(\frac{|\xi| \eta}{2 \tau}\right) d \eta \\
& \leq & \frac{2^a}{\sqrt{\tau}} \int_{\tau}^{\infty} e^{-\frac{(\eta -|\xi|)^2}{4 \tau}} d\eta
\leq \frac{2^a \sqrt{\tau}}{\tau - |\xi|} e^{-\frac{(\tau - |\xi|)^2}{4 \tau}} \leq
\frac{2^{a+1}}{\sqrt{\tau}} e^{-\frac{1}{16}\tau}.
\end{eqnarray*}
Then, we represent
\begin{eqnarray*}
II & = & \frac{1}{\sqrt{\tau}} \int_{-\tau}^{\infty}
(\phi_a(\eta + 2 \tau) - 1) e^{-\frac{\xi^2+\eta^2}{4 \tau}}
\sinh\left(\frac{\xi \eta}{2 \tau}\right) d \eta + \frac{1}{\sqrt{\tau}} \int_{-\tau}^{\infty}
e^{-\frac{\xi^2+\eta^2}{4 \tau}}
\sinh\left(\frac{\xi \eta}{2 \tau}\right) d \eta \\
& \equiv & II_a + II_b,
\end{eqnarray*}
and estimate
\begin{eqnarray*}
|II_b| = \frac{1}{\sqrt{\tau}} \int_{\tau}^{\infty}
e^{-\frac{\xi^2+\eta^2}{4 \tau}}
\sinh\left(\frac{|\xi| \eta}{2 \tau}\right) d \eta \leq
\frac{1}{\sqrt{\tau}} \int_{\tau}^{\infty} e^{-\frac{(\eta - |\xi|)^2}{4 \tau}} d\eta
\leq \frac{2}{\sqrt{\tau}} e^{-\frac{1}{16}\tau}.
\end{eqnarray*}
and
\begin{eqnarray*}
|II_a| \leq \frac{C a}{\sqrt{\tau}} e^{2|\xi|/a -4 \tau/a} \int_{-\tau}^{\infty}
e^{-\frac{(\eta - |\xi| + 4 \tau/a)^2}{4 \tau}} d \eta \leq \frac{C a}{\sqrt{\tau}} \sqrt{\tau}
e^{-3 \tau/a} \leq \frac{C\sqrt{\log(a)}}{\sqrt{\tau} a^{3 \delta + 3\delta^2/2}},
\end{eqnarray*}
where we have used again the fact that the function $\sqrt{\tau} e^{-3 \tau/a}$ is
monotonically decreasing for any $\tau >\frac{a}{6}$, whereas
$\tau_a \gg \mathcal{O}(a)$ as $a \to \infty$.
Because $I$ and $II_b$ are exponentially small in $a$ for all
$\tau \geq \tau_a$, there is $C > 0$ such that
\begin{equation}
\label{bound-first-last}
\sup_{|\xi| \leq \frac{1}{2} \tau} |\psi_+(\xi,\tau) - \psi_-(\xi,\tau)| \leq
\frac{C\sqrt{\log(a)}}{\sqrt{\tau} a^{3 \delta + 3\delta^2/2}} \leq
\frac{C}{\sqrt{\tau} a^{3 \delta + \delta^2}}, \quad \tau \geq \tau_a.
\end{equation}

On the other hand, for any $|\xi| \geq \frac{1}{2} \tau$, there is $C > 0$ such that
\begin{equation}
\label{bound-second-last}
{\rm sech}^2(\xi) | \psi_+(\xi,\tau) - \psi_-(\xi,\tau)| \leq C 2^a e^{-\tau},
\end{equation}
which is exponentially small in $a$ for all $\tau \geq \tau_a$.
It follows from (\ref{bound-first-last})--(\ref{bound-second-last}) that
for any fixed $\delta > 0$ and for sufficiently large $a \gg 1$,
there is constant $C > 0$ such that
\begin{equation}
\label{bound-2new}
\sup_{\xi \in \R} {\rm sech}^2(\xi) | \psi_+(\xi,\tau) - \psi_-(\xi,\tau)| \leq
\frac{C}{\sqrt{\tau} a^{3 \delta + \delta^2}}, \quad \tau \geq \tau_a.
\end{equation}
Representation (\ref{compact-form-new}) and bounds (\ref{bound-1new-Psi}), (\ref{bound-1new}) and (\ref{bound-2new})
yield the desired bound (\ref{main-bound-line-infty-norm}).

We can now prove bound (\ref{main-bound-line-infty-boundary}) for $m = 0$. Using the exponential
factors, we have bound (\ref{bound-third}), which
is exponentially small as $a \to \infty$ for all $\xi \geq \xi_a$
and $\tau \geq 0$. We can rewrite bounds (\ref{bound-third-new})
and (\ref{bound-second-last}) in the equivalent form:
$$
e^{-2\xi} |\partial_{\xi} \psi_-(\xi,\tau)| \leq
C 2^a e^{-2\xi_a}, \quad \xi \geq \xi_a, \quad \tau \geq 0,
$$
and
$$
{\rm sech}^2(\xi) | \psi_+(\xi,\tau) - \psi_-(\xi,\tau)| \leq C 2^a e^{-2 \xi_a},
\quad \xi \geq \xi_a, \quad \tau \geq 0.
$$
Therefore, these bounds are also exponentially small as $a \to \infty$ and
we only need to show that the terms $\psi_+(\xi,\tau)$
and $\partial_{\xi} \psi_+(\xi,\tau)$ remain small for all
$\xi \geq (1 + \delta) \xi_a$ and $\tau \geq 0$. We need the extra factor $(1 + \delta)$
in $\xi \geq (1 + \delta) \xi_a$ to ensure that $z_a$ is bounded from below by
$$
z_a = \frac{\xi-\xi_a+2\tau}{2 \sqrt{\tau}} \geq \frac{\delta \xi_a + 2 \tau}{2 \sqrt{\tau}}
\geq \sqrt{2 \delta \xi_a} = \sqrt{\delta (1 + \delta) a \log(a)},
$$
where we have used the fact that the function $\sqrt{\tau} + \frac{\delta \xi_a}{2 \sqrt{\tau}}$
reaches minimum at $\tau_0 = \frac{\delta \xi_a}{2} > 0$.

Using the splitting of $\psi_+$ in (\ref{splitting-psi}), we infer that
the estimates for $I$ and $III$ produce exponentially small terms in $a$,
whereas the estimate for $II$ produces an algebraically small term in $a$.
As a result of analysis similar to the one for (\ref{bound-first}),
we find that there is $C > 0$ such that
\begin{equation}
\label{bound-first-outer}
\sup_{\tau \geq 0} \sup_{\xi \geq (1 + \delta) \xi_a} |\psi_+(\xi,\tau)| \leq
C a e^{-2 (1+\delta) \xi_a/a} \leq \frac{C}{a^{2 \delta + \delta^2}}.
\end{equation}

Using the exact expression (\ref{derivative-1}) and
the fact that the function $\frac{1}{\sqrt{\tau}} e^{-(\delta \xi_a + 2 \tau)^2/(4 \tau)}$
reaches maximum for $\tau \geq 0$ at
$$
\tau_0 = \frac{1}{4} \left( \sqrt{1 + 4 \delta^2 \xi_a^2} - 1\right) \approx \frac{\delta \xi_a}{2},
$$
we obtain
$$
\partial_{\xi} I \leq \frac{C}{\sqrt{a \log(a)}} e^{-\delta (1 + \delta) a \log(a)}.
$$
Similarly, using the exact expression (\ref{derivative-3}), we obtain
$$
\left| \partial_{\xi} III \right| \leq
\frac{C}{\sqrt{a \log(a)}} e^{-\delta (1 + \delta) a \log(a) + a \log(2)}.
$$
Using the exact expression (\ref{derivative-2}) rewritten as
\begin{eqnarray*}
\partial_{\xi} II = \frac{1}{2 \sqrt{\tau}} e^{-z_a^2} (\phi_a(\xi_a) - 1) +
\int_{-z_a}^{\infty} e^{-z^2} \phi_a'(\xi + 2 \tau + 2 \sqrt{\tau} z) dz,
\end{eqnarray*}
and the bounds (\ref{bound-phi-a}) and (\ref{bound-phi-derivative}), we obtain
\begin{eqnarray*}
\left| \partial_{\xi} II \right| \leq  \frac{C}{\sqrt{\log(a)} a^{\delta + 1/2}} e^{-\delta (1 + \delta) a \log(a)} +
\frac{C}{a^{1+\delta}}.
\end{eqnarray*}
Combining these estimates together, we infer that there is $C > 0$ such that
\begin{equation}
\label{algebraically-small}
\sup_{\tau \geq 0} \sup_{\xi \geq (1 + \delta) \xi_a} |\partial_{\xi} \psi_+(\xi,\tau)|
\leq \frac{C}{a^{1+\delta}}.
\end{equation}
Using (\ref{bound-first-outer}) and (\ref{algebraically-small}),
we obtain the desired bound (\ref{main-bound-line-infty-boundary}) for $m = 0$.
Bound (\ref{main-bound-line-infty-boundary}) for $m \in \N$ follows by recursion from the higher-order derivatives of
the exact solution (\ref{compact-form-new}), the decompositions (\ref{splitting-psi})
and (\ref{splitting-psi-minus}), and the bound (\ref{bound-phi-higher-derivative}) on
the higher-order derivatives of the function $\phi_a$.
\end{proof}

\begin{corollary}
Fix $\delta > 0$ and let $a = \frac{4k}{l}$.
There exist $A > 0$ and $C > 0$ such that for all $a \geq A$,
\begin{equation}
\label{bound-asymptotic-solution-other-case}
\sup_{x \in \R} \left| u(x,t) - u_{\infty}(x) \right| \leq \frac{C}{a^{3 \delta + \delta^2} t^{1/2}}, \quad
\mbox{\rm for all} \;\; t \geq T_* := \frac{(1 + \delta)^2 a \log(a)}{32 k^2}.
\end{equation}
\label{corollary-2}
\end{corollary}

\begin{proof}
Bound (\ref{bound-asymptotic-solution-other-case}) follows from (\ref{main-bound-line-infty-norm})
after the scaling transformation (\ref{variables-hear}).
\end{proof}

\begin{remark}
We can also prove convergence of $u$ to $u_{\infty}$ in $H^1$ norm but
it requires some decay as $|\xi| \to \infty$ and an extension of
bound (\ref{main-bound-line-infty-norm}) for the first derivative in $\xi$,
which we do not establish in Lemma \ref{lemma-instant-solution}.
Note that bounds (\ref{main-bound-line-infty-norm})
and (\ref{main-bound-line-infty-boundary})
will be applied on large but finite intervals in $\xi$.
\end{remark}

Figure \ref{figure-3} shows the exact solution $w(\xi,\tau)$ with $a = 10$ for
different values of $\tau$ together with
the rescaled enstrophy $E = \int_0^{\infty} w_{\xi}^2 d\xi$. The integrals
in the exact solution (\ref{compact-form-new})
were approximated by using the MATLAB \verb"quad" function. The behavior of $w$ and $E$ looks similar
to the case $a = 4$ shown on Figure \ref{figure-2} but it takes longer for
$E$ to approach to the limit $E_{\infty} = \frac{2}{3}$ from the initially
smaller value $E_0 = \frac{2}{3a}$.

Explicit computation with the initial data (\ref{initial-data-inf-modified}) yields
$E(u_0) = \frac{32}{3} k^2 l$, whereas we recall that $E(u_{\infty}) = \frac{128}{3} k^3$.
If $l = \mathcal{O}(1)$ as $k \to \infty$,
we have $k = \mathcal{O}(\mathcal{E}^{1/2})$ and hence
\begin{equation}
\label{growth-2}
E(u_{\infty}) = \mathcal{O}(\mathcal{E}^{3/2}), \quad
T_* = \mathcal{O}({\cal E}^{-1/2} \log({\cal E})), \quad \mbox{\rm as} \quad
\mathcal{E} \to \infty.
\end{equation}

\begin{figure}
\begin{center}
  \includegraphics[width=0.48\textwidth]{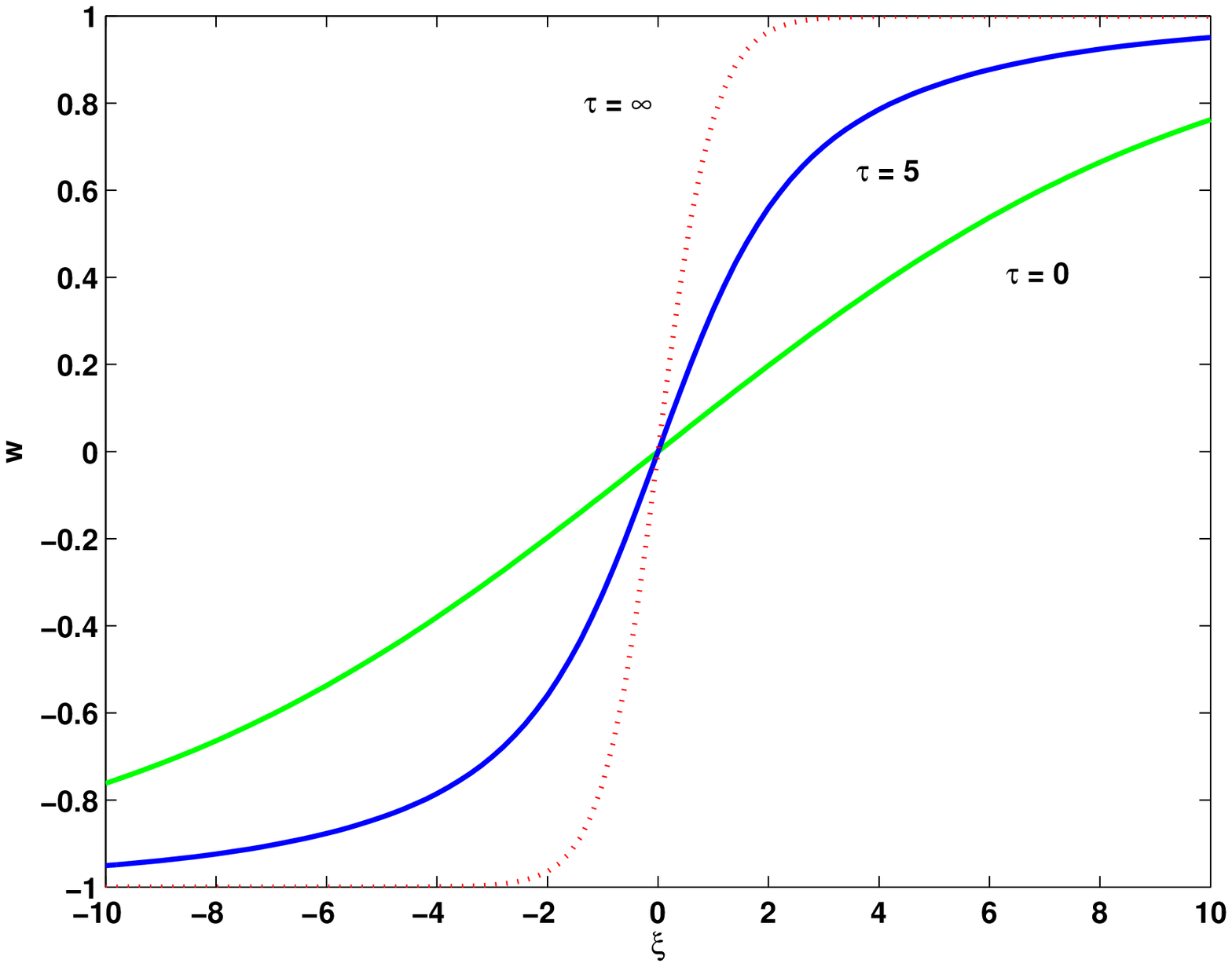}
  \includegraphics[width=0.48\textwidth]{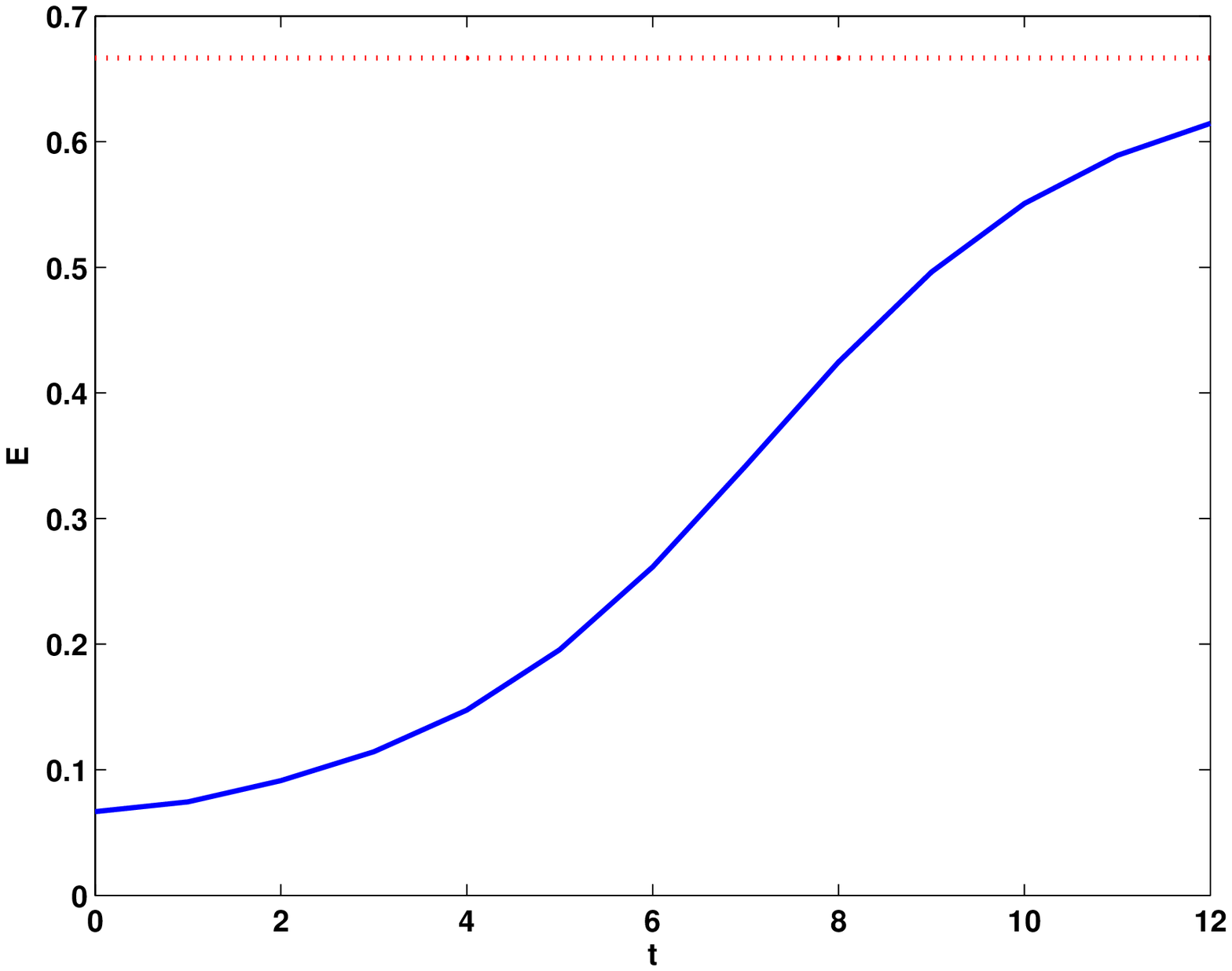}
\end{center}
\caption{Left: The exact solution (\ref{compact-form-new}) with $a = 10$
for different values of $\tau$ and in the limit $\tau \to \infty$.
Right: Enstrophy $E$ versus time $\tau$. The dotted
curve shows the value of $E$ at the viscous shock $w_{\infty}(\xi) = \tanh(\xi)$.}
\label{figure-3}
\end{figure}

If $l = (1+\Delta) \log(k)$, then Lemma \ref{lemma-exp-asymptotics} implies
that $k = \mathcal{O}(\mathcal{E}^{1/2} \log^{-1/2}(\mathcal{E}))$.
From (\ref{bound-asymptotic-solution-other-case}), we have
\begin{equation}
\label{definition-T-star}
T_* = \frac{(1 + \delta)^2}{8 (1+\Delta) k} \left[ 1 +
\mathcal{O}\left( \frac{\log\log(k)}{\log(k)}\right) \right].
\end{equation}
Hence we have
\begin{equation}
\label{growth-3}
E(u_{\infty}) = \mathcal{O}(\mathcal{E}^{3/2} \log^{-3/2}(\mathcal{E}))  \quad
\mbox{\rm and} \quad
T_* = \mathcal{O}({\cal E}^{-1/2} \log^{1/2}({\cal E})) \quad \mbox{\rm as} \quad
\mathcal{E} \to \infty.
\end{equation}
The asymptotic result (\ref{growth-3}) is included in the scaling law
(\ref{arg-max-theorem}) of Theorem \ref{theorem-time}. On the other hand,
the maximal rate (\ref{growth-2}) can not be achieved,
because if $l = \mathcal{O}(1)$, then $a = \mathcal{O}(k)$ and
the lower bound on the time $\tau \geq  \frac{1}{2} (1 + \delta)^2 a \log(a)$ in
bound (\ref{main-bound-line-infty-norm}) is beyond
the time interval $\tau \in (0,k)$ in Lemma \ref{lemma-solution-Burgers}
as $k \to \infty$.

\section{Dynamics of a viscous shock in a bounded domain}

By Lemma \ref{lemma-solution-Burgers}, the initial-value problem (\ref{Burgers})
is replaced by the equivalent problem,
\begin{equation}
\label{Burgers-rescaled-first}
\left\{ \begin{array}{l} w_{\tau} = 2 w w_{\xi} + w_{\xi \xi}, \quad \quad
|\xi| < 2(k - \tau), \quad \tau \in (0,k), \\ w|_{\tau = 0} = w_0, \quad \quad \quad  \quad \;\;
|\xi| \leq 2k,\end{array} \right.
\end{equation}
subject to the initial condition $w_0(\xi) = f(\xi/4k)$ and
the inhomogeneous boundary conditions $w = \pm 1$ at $\xi = \pm 2 (k - \tau)$, where
$f$ is given by either (\ref{instant-function}) or
(\ref{particular-function}).

Let us define the norm
\begin{equation}
\label{norms-tau}
\| w \|_{L^2_{k,\tau}} := \left( \int_{-2(k-\tau)}^{2(k-\tau)} |w(\xi,\tau)|^2 d \xi \right)^{1/2}.
\end{equation}
An approximate solution of the
initial-value problem (\ref{Burgers-rescaled-first}) can be thought in the form,
\begin{equation}
\label{approx-solution-w-1}
w_{\rm app}(\xi,\tau) = \frac{\partial}{\partial \xi} \log \psi_{\rm app}(\xi,\tau),
\end{equation}
where $\psi_{\rm app}$ is a solution of the homogeneous heat equation
on the real line. For $f$ in either (\ref{instant-function}) or
(\ref{particular-function}), the solution is constructed in Lemmas \ref{lemma-instant-solution}
or \ref{lemma-Burgers-solution}, respectively. We shall assume that
\begin{equation}
\label{positivity-slope}
0 \leq \frac{\partial}{\partial \xi} w_{\rm app}(\xi,\tau) \leq 1,
\quad \xi \in \R, \quad \tau \in \R_+,
\end{equation}
which is a reasonable assumption for a monotonic transition from
$w_{\rm app,0}(\xi) = \tanh(\xi/a)$ with either $a = 4$ or $a \to \infty$
as $k \to \infty$ to the viscous shock $w_{\infty}(\xi) = \tanh(\xi)$.

We write $w_{\rm app}(\xi,\tau) = \tanh(\xi) + \tilde{w}(\xi,\tau)$
and assume that for any fixed $C_0 \in (0,1)$, and for sufficiently large $k \gg 1$,
there is a small $C_k > 0$ such that
\begin{equation}
\label{approx-solution-w-3}
\sup_{\tau \in [0,C_0 k]} \left( |\tilde{w}(2(k - \tau),\tau)| +
\left| \frac{d}{d \tau} \tilde{w}(2(k - \tau),\tau) \right| +
\left| \frac{d^2}{d \tau^2} \tilde{w}(2(k - \tau),\tau) \right| \right) \leq
\frac{C_k}{k},
\end{equation}
where $C_k \to 0$ as $k \to \infty$.

Furthermore, we assume that for sufficiently large $k \gg 1$, there is
a small $D_k > 0$ such that
\begin{equation}
\label{approx-solution-w-3-d}
\sup_{\xi \in [-2k,2k]} \left(
\left| w_0(\xi) - w_{{\rm app},0}(\xi) \right| +
\left| w_0'(\xi) - w_{{\rm app},0}'(\xi) \right| \right) \leq
\frac{D_k}{k^{3/2}},
\end{equation}
where $D_k \to 0$ as $k \to \infty$.

Using the Cole--Hopf transformation, we rewrite the initial-value problem (\ref{Burgers-rescaled-first})
in the form,
\begin{equation}
\label{Burgers-rescaled-second}
w = \frac{\psi_{\xi}}{\psi} \quad \Rightarrow \quad
\left\{ \begin{array}{l} \psi_{\tau} = \psi_{\xi \xi}, \quad \quad \;\;
|\xi| < 2(k - \tau), \quad \tau \in (0,k), \\
\psi|_{\tau = 0} = \psi_0, \quad \;\; |\xi| \leq 2k,\end{array} \right.
\end{equation}
subject to the initial condition $\psi_0 = \exp\left(\int_0^{\xi} w_0(\xi') d \xi'\right)$ and
the Robin boundary conditions $\psi_{\xi} = \pm \psi$ at $\xi = \pm 2 (k - \tau)$.
Using the decomposition
\begin{equation}
\label{substition-psi-approx}
\psi = \psi_{\rm app} (1 + \Psi),
\end{equation}
we find the equivalent initial-value problem,
\begin{equation}
\label{Burgers-rescaled-third}
\left\{ \begin{array}{l} \Psi_{\tau} = \Psi_{\xi \xi} + 2 w_{\rm app} \Psi_{\xi},
\quad \quad \quad |\xi| < 2(k - \tau), \quad \tau \in (0,k), \\
\Psi|_{\tau = 0} = \Psi_0, \quad \quad \quad \quad \quad \quad \quad |\xi| \leq 2k,\end{array} \right.
\end{equation}
subject to the initial condition,
\begin{equation}
\label{Psi-zero}
\Psi_0(\xi) = \frac{\psi_0(\xi)}{\psi_{{\rm app},0}(\xi)} - 1 = \exp\left( \int_0^{\xi}
\left[ w_0(\xi') - w_{{\rm app},0}(\xi') \right] d \xi'\right) - 1,
\end{equation}
and the Robin inhomogeneous boundary conditions,
\begin{equation}
\label{boundary-rescaled-third}
\Psi_{\xi} = \pm \chi(\tau) (1+\Psi), \quad \xi = \pm 2 (k - \tau), \quad \tau \in [0,k],
\end{equation}
where
\begin{equation}
\label{boundary-rescaled-expression}
\chi(\tau) = 1 - w_{\rm app}(2(k-\tau),\tau) =
\frac{2 e^{-4(k-\tau)}}{1 + e^{-4(k-\tau)}} - \tilde{w}(2(k-\tau),\tau).
\end{equation}

By bound (\ref{approx-solution-w-3}) and explicit expression (\ref{boundary-rescaled-expression}),
for any fixed $C_0 \in (0,1)$ and for sufficiently large $k \gg 1$, there is a small $C_k > 0$ such that
\begin{equation}
\label{chi-estimate}
\sup_{\tau \in [0,C_0 k]} \left( |\chi(\tau)| + |\chi'(\tau)| +
|\chi''(\tau)| \right) \leq \frac{C_k}{k}.
\end{equation}
Because $C_k \to 0$ as $k \to \infty$ in (\ref{chi-estimate}),
the Robin boundary conditions (\ref{boundary-rescaled-third}) converge to
the Neumann conditions as $k \to \infty$.

By bound (\ref{approx-solution-w-3-d}) and explicit expression (\ref{Psi-zero}),
for sufficiently large $k \gg 1$, there is a small $D_k > 0$ such that
\begin{equation}
\label{initial-estimate-small}
\sup_{\xi \in [-2k,2k]} |\Psi_0(\xi)| \leq \frac{D_k}{k^{1/2}}, \quad
\| \Psi_0 \|_{L^2_{k,0}} \leq D_k, \quad
\| \Psi_0' \|_{L^2_{k,0}} + \| \Psi_0'' \|_{L^2_{k,0}} \leq \frac{D_k}{k}.
\end{equation}
Because $D_k \to 0$ as $k \to \infty$ in (\ref{initial-estimate-small}),
the initial condition (\ref{Psi-zero}) converges to $0$ as $k \to \infty$.

Using apriori energy estimates, we prove the following result.

\begin{lemma}
Assume (\ref{positivity-slope}), (\ref{chi-estimate}), and (\ref{initial-estimate-small}).
Fix $C_0 \in (0,1)$. There exist constants $K \geq 1$ and $C > 0$, such that for all
$k \geq K$ and $\tau \in [0,C_0 k]$, the unique solution
of the Burgers equation (\ref{Burgers-rescaled-first}) satisfies
\begin{equation}
\label{bound-w-main}
\| w - w_{\rm app} \|^2_{L^2_{k,\tau}} + \| \partial_{\xi} (w - w_{\rm app}) \|^2_{L^2_{k,\tau}}
\leq C \left( C_k + D_k^2 \right)
\end{equation}
where $C_k$ and $D_k$ are defined by (\ref{chi-estimate}) and (\ref{initial-estimate-small}).
\label{lemma-new-approximation}
\end{lemma}

\begin{proof}
We note the correspondence,
\begin{equation}
w = w_{\rm app} + \frac{\Psi_{\xi}}{1 + \Psi}.
\end{equation}
Fix $C_0 \in (0,1)$ and introduce the energy for the initial-boundary value problem
(\ref{Burgers-rescaled-third}),
\begin{equation}
H(\tau) := \frac{1}{2} \| \Psi \|_{L^2_{k,\tau}}^2 + \frac{1}{2} \| \Psi_{\xi} \|_{L^2_{k,\tau}}^2,
\quad \tau \in [0,C_0 k].
\end{equation}
If $H(\tau)$ is small for all $\tau \in [0,C_0 k]$, then there is $C > 0$
such that
\begin{equation}
\label{correspondence-psi}
\| w - w_{\rm app} \|^2_{L^2_{k,\tau}} \leq C H(\tau).
\end{equation}

The initial-boundary value problem (\ref{Burgers-rescaled-third})
inherits the local well-posedness of the initial-value problem
(\ref{Burgers}) if $\Psi_0 \in H^1([-2k,2k])$ as long as
\begin{equation}
\label{constraint-L-infty}
\sup_{|\xi| \leq 2(k-\tau)} |\Psi(\xi,\tau)| < 1, \quad \tau \in [0, C_0 k],
\end{equation}
to ensure that $1 + \Psi(\xi,\tau) > 0$ for all
$|\xi| \leq 2 (k-\tau)$ and $\tau \in [0,k]$. Recalling
Sobolev's embedding of $H^1$ to $L^{\infty}$, we have
\begin{equation}
\label{Sobolev-embedding}
\exists C > 0 : \quad \sup_{|\xi| \leq 2(k-\tau)} |\Psi(\xi,\tau)| \leq C \sqrt{H(\tau)},
\quad \tau \in [0,C_0 k],
\end{equation}
hence the constraint (\ref{constraint-L-infty}) is satisfied
if $H(\tau)$ remains small for all $\tau \in [0, C_0 k]$.

Computations are simplified for the even solutions $\Psi(-\xi,\tau) = \Psi(\xi,\tau)$,
which generate odd solutions $w(-\xi,\tau) = -w(\xi,\tau)$ of the Burgers
equation (\ref{Burgers-rescaled-first}). Multiplying the heat equation
(\ref{Burgers-rescaled-third}) by the solution $\Psi$
and integrating in $\xi$ over the time-dependent interval
$[-\xi_+(\tau),\xi_+(\tau)]$ with $\xi_+(\tau) = 2(k-\tau)$,
we obtain
\begin{eqnarray*}
\frac{1}{2}\frac{d}{d \tau} \| \Psi \|_{L^2_{k,\tau}}^2 + 2 \Psi_+^2(\tau) & = &
2 \chi(\tau) \Psi_+(\tau) (1 + \Psi_+(\tau)) +
2 w_{\rm app}(\xi_+(\tau),\tau) \Psi_+^2(\tau) \\
& \phantom{t} & - \| \Psi_{\xi} \|_{L^2_{k,\tau}}^2 - \| w_{{\rm app},\xi}^{1/2} \Psi \|_{L^2_{k,\tau}}^2,
\end{eqnarray*}
where $\Psi_+(\tau) = \Psi(\xi_+(\tau),\tau)$. Canceling
redundant terms, this expression becomes
\begin{eqnarray}
\label{energy-rate-one}
\frac{1}{2} \frac{d}{d \tau} \| \Psi \|_{L^2_{k,\tau}}^2 =
2 \chi(\tau) \Psi_+(\tau) - \| \Psi_{\xi} \|_{L^2_{k,\tau}}^2 - \|
w_{{\rm app},\xi}^{1/2} \Psi \|_{L^2_{k,\tau}}^2.
\end{eqnarray}
Thanks to property (\ref{positivity-slope}), the rate of change of
$\| \Psi \|_{L^2_{k,\tau}}^2$ in $\tau$ is almost negative definite, except of the first
boundary term, which is small as $k \to \infty$. The boundary term is
controlled by (\ref{chi-estimate}) and (\ref{Sobolev-embedding}). We need one more equation
to be able to control the energy $H(\tau)$ for $\tau \in [0,C_0 k]$.

Differentiating the heat equation (\ref{Burgers-rescaled-third}) in $\xi$,
multiplying the resulting equation by $\Psi_{\xi}$,
integrating in $\xi$ over
$[-\xi_+(\tau),\xi_+(\tau)]$, and performing similar simplifications, we obtain
\begin{eqnarray*}
\frac{1}{2} \frac{d}{d \tau} \| \Psi_{\xi} \|_{L^2_{k,\tau}}^2 & = &
2 \chi(\tau) (1 + \Psi_+(\tau)) (\Psi_{\xi \xi})_+(\tau) -
2 \chi^3(\tau) (1 + \Psi_+(\tau))^2 \\
& \phantom{t} & - \| \Psi_{\xi \xi} \|_{L^2_{k,\tau}}^2 + \|
w_{{\rm app},\xi}^{1/2} \Psi_{\xi} \|^2_{L^2_{k,\tau}},
\end{eqnarray*}
where $(\Psi_{\xi \xi})_+(\tau) = \Psi_{\xi \xi}(\xi_+(\tau),\tau)$.
For strong solutions of the initial-value problem (\ref{Burgers-rescaled-third}),
we obtain by continuity that
\begin{equation}
\label{first-der}
\frac{d}{d \tau} \Psi_+(\tau) = \left( \Psi_{\tau} - 2 \Psi_{\xi} \right) \biggr|_{\xi = \xi_+(\tau)}
= (\Psi_{\xi \xi})_+(\tau) - 2 \chi^2(\tau) (1 + \Psi_+(\tau)).
\end{equation}
Hence we have
\begin{eqnarray}
\nonumber
\frac{1}{2} \frac{d}{d \tau} \| \Psi_{\xi} \|_{L^2_{k,\tau}}^2 & = &
2 \chi(\tau) (1 + \Psi_+(\tau)) \frac{d}{d \tau} \Psi_+(\tau) + 2 \chi^3(\tau) (1 + \Psi_+(\tau))^2 \\
\label{energy-rate-two} & \phantom{t} & \phantom{texttext}
- \| \Psi_{\xi \xi} \|_{L^2_{k,\tau}}^2 + \|
w_{{\rm app},\xi}^{1/2} \Psi_{\xi} \|^2_{L^2_{k,\tau}},
\end{eqnarray}

The positive last term in (\ref{energy-rate-two}) is compensated
by the negative second term in (\ref{energy-rate-one}) in the sum of these two expressions.
Additionally, we can move the derivative of $\Psi_+(\tau)$ under the derivative sign
and obtain
\begin{eqnarray}
\nonumber
\frac{d}{d \tau} H_1(\tau) =
2 ( \chi(\tau) - \chi'(\tau) ) \Psi_+(\tau)  - \chi'(\tau) \Psi_+^2(\tau)
+ 2 \chi^3(\tau) (1 + \Psi_+(\tau))^2 \\ \label{energy-derivative}
- \| \Psi_{\xi \xi} \|_{L^2_{k,\tau}}^2 - \| w_{{\rm app},\xi}^{1/2} \Psi \|^2_{L^2_{k,\tau}}
- \| (1 - w_{{\rm app},\xi})^{1/2} \Psi_{\xi} \|^2_{L^2_{k,\tau}},
\end{eqnarray}
where
$$
H_1(\tau) := H(\tau) - \chi(\tau) (2 \Psi_+(\tau) + \Psi_+^2(\tau)).
$$
The last three terms in the right-hand-side of equation (\ref{energy-derivative})
are negative thanks to property (\ref{positivity-slope}).
On the other hand, the functions $\chi(\tau)$ and $\chi'(\tau)$ are controlled by
(\ref{chi-estimate}). Integrating
(\ref{energy-derivative}) on $[0,C_0 \tau]$ and using Sobolev's inequality (\ref{Sobolev-embedding})
and an elementary inequality $2 \sqrt{H(\tau)} \leq 1 + H(\tau)$,
we obtain
\begin{eqnarray*}
H_1(\tau) - H_1(0) \leq \frac{C_k}{k} \int_0^{\tau} (1 + H(\tau')) d \tau',
\quad \tau \in [0,C_0k].
\end{eqnarray*}
By the estimate (\ref{chi-estimate}), Sobolev's inequality (\ref{Sobolev-embedding}) again,
and Gronwall's inequality, we infer that there is a constant $C > 0$ such that
\begin{equation}
\label{control-first-level}
H(\tau) \leq C \left( H(0) + C_k \right), \quad \tau \in [0,C_0 k].
\end{equation}
By the estimate (\ref{initial-estimate-small}), we have $H(0) \leq D_k^2$ and
bound (\ref{control-first-level})  yields control of the first term
in the bound (\ref{bound-w-main}) from the correspondence (\ref{correspondence-psi}).

Let us now control the second term in the bound (\ref{bound-w-main}).
Differentiating the heat equation (\ref{Burgers-rescaled-third}) twice in $\xi$,
multiplying the resulting equation by $\Psi_{\xi \xi}$,
integrating in $\xi$ over $[-\xi_+(\tau),\xi_+(\tau)]$,
and performing similar simplifications, we obtain
\begin{eqnarray}
\nonumber
\frac{1}{2} \frac{d}{d \tau} \| \Psi_{\xi \xi} \|_{L^2_{k,\tau}}^2 & = &
2 (\Psi_{\xi \xi})_+(\tau) (\Psi_{\xi \xi \xi})_+(\tau) -
2 \chi(\tau) (\Psi_{\xi \xi})^2_+(\tau) + 2 (w_{{\rm app},\xi\xi})_+
(\Psi_{\xi})^2_+(\tau) \\
\label{energy-rate-three}
& \phantom{t} & - \| \Psi_{\xi \xi \xi} \|_{L^2_{k,\tau}}^2 + 3 \|
w_{{\rm app},\xi}^{1/2} \Psi_{\xi \xi} \|^2_{L^2_{k,\tau}} -
\| w_{{\rm app},\xi \xi \xi}^{1/2} \Psi_{\xi} \|^2_{L^2_{k,\tau}},
\end{eqnarray}
where the subscript ``$+$'' denotes the boundary value at $\xi = \xi_+(\tau)$.
For strong solutions of the initial-value problem (\ref{Burgers-rescaled-third}),
we obtain by continuity that
\begin{eqnarray}
\nonumber
\frac{d}{d \tau} (\Psi_{\xi})_+(\tau) & = & \left( \Psi_{\tau \xi} - 2 \Psi_{\xi \xi} \right) \biggr|_{\xi = \xi_+(\tau)} \\ \label{second-der}
& = & (\Psi_{\xi \xi \xi})_+(\tau) - 2 \chi(\tau) (\Psi_{\xi \xi})_+(\tau)
+ 2 (w_{{\rm app},\xi})_+ (\Psi_{\xi})_+(\tau)
\end{eqnarray}
and
\begin{equation}
\label{third-der}
(w_{{\rm app},\xi \xi})_+ = (w_{{\rm app},\tau})_+ - 2 (w_{\rm app})_+ (w_{{\rm app},\xi})_+
= 2 \chi(\tau) (w_{{\rm app},\xi})_+ - \chi'(\tau).
\end{equation}

Using (\ref{boundary-rescaled-third}), (\ref{first-der}), (\ref{second-der}), and (\ref{third-der})
we can rewrite (\ref{energy-rate-three}) in the form:
\begin{eqnarray*}
\frac{1}{2} \frac{d}{d \tau} \| \Psi_{\xi \xi} \|_{L^2_{k,\tau}}^2 & = &
- 2 (2 \chi^3(\tau)
+ 2 \chi(\tau) (w_{{\rm app},\xi})_+ - \chi'(\tau))
(1 + \Psi_+(\tau)) \frac{d}{d \tau} \Psi_+(\tau)   \\
& \phantom{t} &
- 2 \chi^2(\tau) (1 + \Psi_+(\tau))^2 (4 \chi^3(\tau)
+ 2 \chi(\tau) (w_{{\rm app},\xi})_+ - \chi'(\tau)) \\
& \phantom{t} & + 4 \chi(\tau) (\Psi_{\xi \xi})^2_+(\tau) - \| \Psi_{\xi \xi \xi} \|_{L^2_{k,\tau}}^2 + 3 \|
w_{{\rm app},\xi}^{1/2} \Psi_{\xi \xi} \|^2_{L^2_{k,\tau}} \\
& \phantom{t} & -
\| w_{{\rm app},\xi \xi \xi}^{1/2} \Psi_{\xi} \|^2_{L^2_{k,\tau}}.
\end{eqnarray*}

To control $\| \Psi_{\xi \xi} \|_{L^2_{k,\tau}}^2$, we define
\begin{eqnarray}
\nonumber
H_2(\tau) & := &  \frac{1}{2} \| \Psi_{\xi \xi} \|_{L^2_{k,\tau}}^2 + 4 H_1(\tau)
+ \frac{\gamma}{2} \| \Psi \|_{L^2_{k,\tau}}^2 \\ \label{energy-last-estimates}
& \phantom{t} &
+ (2 \chi^3(\tau) + 2 \chi(\tau) (w_{{\rm app},\xi})_+ - \chi'(\tau))
\left( 2 \Psi_+(\tau) + \Psi_+^2(\tau) \right),
\end{eqnarray}
where $\gamma$ is to be specified below. Thus we obtain
\begin{eqnarray*}
\frac{d}{d \tau} H_2(\tau) & = &
2 (\Psi_+(\tau) + \Psi_+^2(\tau)) \frac{d}{d \tau} (2 \chi^3(\tau)
+ 2 \chi(\tau) (w_{{\rm app},\xi})_+ - \chi'(\tau))  \\
& \phantom{t} &
- 2 \chi^2(\tau) (1 + \Psi_+(\tau))^2 (4 \chi^3(\tau)
+ 2 \chi(\tau) (w_{{\rm app},\xi})_+ - \chi'(\tau)) \\
& \phantom{t} & + 8 ( \chi(\tau) - \chi'(\tau) ) \Psi_+(\tau)  - 4 \chi'(\tau) \Psi_+^2(\tau)
+ 8 \chi^3(\tau) (1 + \Psi_+(\tau))^2 \\
& \phantom{t} & + 2 \gamma \chi(\tau) \Psi_+(\tau)
+ 4 \chi(\tau) (\Psi_{\xi \xi})^2_+(\tau) - \| \Psi_{\xi \xi \xi} \|_{L^2_{k,\tau}}^2
- \| \Psi_{\xi \xi} \|_{L^2_{k,\tau}}^2 \\
& \phantom{t} & - 3 \|
(1 - w_{{\rm app},\xi})^{1/2} \Psi_{\xi \xi} \|^2_{L^2_{k,\tau}} - 4 \| w_{{\rm app},\xi}^{1/2} \Psi \|^2_{L^2_{k,\tau}}
- 4 \| (1 - w_{{\rm app},\xi})^{1/2} \Psi_{\xi} \|^2_{L^2_{k,\tau}}\\
& \phantom{t} & -
\gamma \| w_{{\rm app},\xi}^{1/2} \Psi \|_{L^2_{k,\tau}}^2 -
\| ( \gamma + w_{{\rm app},\xi \xi \xi})^{1/2} \Psi_{\xi} \|^2_{L^2_{k,\tau}}.
\end{eqnarray*}

If we choose
$$
\gamma := - \inf_{\tau \in [0,C_0 k]} \inf_{\xi \in [-\xi_+(\tau),\xi_+(\tau)]} w_{{\rm app},\xi \xi \xi} \geq 0,
$$
the last term is negative for any $\tau \in [0,C_0 k]$. By Sobolev's inequality,
$$
\exists C > 0 : \quad \sup_{|\xi| \leq 2(k-\tau)} |\Psi_{\xi \xi}(\xi,\tau)|^2 \leq
C (\| \Psi_{\xi \xi \xi} \|_{L^2_{k,\tau}}^2 + \| \Psi_{\xi \xi} \|_{L^2_{k,\tau}}^2), \quad \tau \in [0,C_0 k],
$$
and the smallness of $\chi$, the term
$$
4 \chi(\tau) (\Psi_{\xi \xi})^2_+(\tau) - \| \Psi_{\xi \xi \xi} \|_{L^2_{k,\tau}}^2
- \| \Psi_{\xi \xi} \|_{L^2_{k,\tau}}^2
$$
is also negative. All other integral terms are negative, whereas the
boundary terms are controlled by
the estimate (\ref{chi-estimate}), Sobolev's inequality (\ref{Sobolev-embedding}),
and the previous estimate (\ref{control-first-level}). As a result, we obtain
$$
H_2(\tau) \leq H_2(0) + \frac{C_k}{k} \int_0^{\tau} (1 + H_1(\tau')) d \tau' \leq H_2(0) + C_k,
\quad \tau \in [0,C_0k].
$$
Using (\ref{initial-estimate-small}), (\ref{control-first-level}), and (\ref{energy-last-estimates}),
we infer that there is a constant $C > 0$ such that
\begin{eqnarray}
\label{control-second-level}
\| \Psi_{\xi \xi} \|_{L^2_{k,\tau}}^2 \leq C (C_k + D_k^2), \quad \tau \in [0,C_0k].
\end{eqnarray}
Thanks to the exact expression
$$
\partial_{\xi} (w - w_{\rm app}) = \frac{\Psi_{\xi \xi}}{1 + \Psi} - \frac{\Psi_{\xi}^2}{(1 + \Psi)^2},
$$
and Sobolev's inequality,
$$
\exists C > 0 : \quad \sup_{|\xi| \leq 2(k-\tau)} |\Psi_{\xi}(\xi,\tau)|^2 \leq
C (\| \Psi_{\xi \xi} \|_{L^2_{k,\tau}}^2 + \| \Psi_{\xi} \|_{L^2_{k,\tau}}^2), \quad \tau \in [0,C_0 k],
$$
bounds (\ref{control-first-level}) and (\ref{control-second-level})
yield control of the second term in the bound (\ref{bound-w-main}).
\end{proof}

\begin{remark}
Recall that $w_{\rm app}(\xi,\tau) = \tanh(\xi) + \tilde{w}(\xi,\tau)$, where
$\tilde{w}$ is small for large values of $\tau$. If we consider the operator
$L = \partial_{\xi}^2 + 2 \tanh(\xi) \partial_{\xi}$ on the truncated domain
$[-\xi_0,\xi_0]$ subject to the Neumann boundary conditions,
then the eigenvalues of this boundary-value problem for even eigenfunctions
are located at the real axis and bounded from above by $-1$ for
any $\xi_0 > 0$. This suggests the asymptotic stability of the zero solution
of $\Psi_{\tau} = L \Psi$ but does not imply any good bounds on the resolvent
operator $(\lambda {\rm Id} - L)^{-1}$, which is needed for the estimates of the
remainder terms. Moreover, the resolvent operator $(\lambda {\rm Id} - L)^{-1}$
may grow exponentially as $\xi_0 \to \infty$ because the continuous spectrum
of $L$ on the infinite line domain $\R$ touches the imaginary axis and the zero eigenvalue.
See Section 4.4 in Scheel \& Sandstede \cite{Scheel}. Apriori energy estimates
used in the proof of Lemma \ref{lemma-new-approximation} avoid this problem,
as well as they incorporate moving boundary conditions in analysis
of the remainder terms.
\end{remark}

\section{Proof of Theorem \ref{theorem-time}: Case $l = \mathcal{O}(k)$ as $k \to \infty$.}

We shall consider the initial data (\ref{initial-data}) and (\ref{instant-function}). This
corresponds to the choice $l = k$, which represents a more general
case $l = \mathcal{O}(k)$ as $k \to \infty$.

By Lemma \ref{lemma-solution-Burgers}, a solution of the Burgers equation (\ref{Burgers})
is written in the form
\begin{equation}
u(x,t) = p(t) \left( 2x - w(p(t) x,\tau) \right), \quad p(t) =
\frac{4 k}{1 + 16 k t}, \quad \tau = \frac{16 k^2 t}{1 + 16 k t},
\end{equation}
where $w$ solves the rescaled Burgers equation (\ref{Burgers-rescaled-first})
with the initial data
\begin{equation}
w_0(\xi) = \lambda \tanh(\xi/4) \quad \Rightarrow \quad
\psi_0(\xi) = \cosh^{4 \lambda}(\xi/4),
\end{equation}
where
$$
\lambda =  \coth(k/2) = 1 + \frac{2 e^{-k}}{1 - e^{-k}}.
$$
The approximate solution of the rescaled Burgers equation (\ref{Burgers-rescaled-first})
is given by Lemma \ref{lemma-instant-solution}. It can be written in the
Cole--Hopf form (\ref{approx-solution-w-1}) with
\begin{equation}
\psi_{\rm app}(\xi,\tau) = \frac{1}{8} \left( 3 + 4 \cosh(\xi/2) e^{\tau/4} + \cosh(\xi) e^{\tau} \right).
\end{equation}
Note that $\psi_{{\rm app},0}(\xi) = \cosh^4(\xi/4)$.
Assumption (\ref{positivity-slope}) is satisfied by the direct computations.
Assumption (\ref{approx-solution-w-3}) follows from bound (\ref{asymptotic-solution}) of Lemma \ref{lemma-instant-solution} with
\begin{equation}
\label{C-k-exponential}
\exists C > 0 : \quad C_k = C k e^{-(k-\tau)}, \quad \tau \in [0,C_0 k],
\end{equation}
which is exponentially small as $k \to \infty$ for any $C_0 \in (0,1)$.

The initial condition,
$$
\Psi_0(\xi) = \cosh^{4 (\lambda - 1)}(\xi/4) - 1 = \exp\left(\frac{8 e^{-k}}{1 - e^{-k}}
\log\left[\cosh\left(\frac{\xi}{4}\right)\right]\right) - 1,
$$
implies that for sufficiently large $k \gg 1$, there is $C > 0$ such that
$$
\sup_{|\xi| \leq 2 k} |\Psi_0(\xi)| \leq C k e^{-k}, \quad
\| \Psi_0 \|_{L^2_{k,0}} \leq C k^{3/2} e^{-k}, \quad
\| \Psi_0' \|_{L^2_{k,0}} + \| \Psi_0'' \|_{L^2_{k,0}} \leq C k^{1/2} e^{-k}.
$$
Hence assumption (\ref{initial-estimate-small}) holds with
$D_k = C k^{3/2} e^{-k}$ for some $C > 0$. Because $D_k^2$ is much smaller than
$C_k$ for all $\tau \in [0,C_0 k]$, Lemma \ref{lemma-new-approximation} yields
\begin{equation}
\label{bound-w-main-new}
\| w - w_{\rm app} \|^2_{L^2_{k,\tau}} + \| \partial_{\xi} \left( w - w_{\rm app} \right)
\|^2_{L^2_{k,\tau}} \leq C k e^{-(k-\tau)}, \quad \tau \in [0,C_0 k].
\end{equation}

Applying Lemma \ref{lemma-solution-Burgers}, we write an approximate
solution of the Burgers equation (\ref{Burgers}) with initial data
(\ref{initial-data}) and (\ref{instant-function}) in the form,
\begin{equation}
\label{aprox-solution}
u_{\rm app}(x,t) = p(t) \left( 2x - w_{\rm app}(p(t) x,\tau) \right), \quad p(t) =
\frac{4 k}{1 + 16 k t}, \quad \tau = \frac{16 k^2 t}{1 + 16 k t},
\end{equation}
on the time interval
\begin{equation}
\label{time-interval}
0 \leq 16 k t \leq \frac{C_0}{1 - C_0}.
\end{equation}
Because $|p(t)| = \mathcal{O}(k)$
as $k \to \infty$ for any $t$ in the time interval (\ref{time-interval}),
bound (\ref{bound-w-main-new}) and Sobolev embedding of $H^1$ to $L^{\infty}$
imply that there are constants $C > 0$ and $\alpha > 0$ such that
\begin{eqnarray}
\label{aprox-solution-bound}
\sup_{x \in \mathbb{T}} \left| u(x,t) - u_{\rm app}(x,t) \right| \leq
C k \sup_{|\xi| \leq 2 (k - \tau)} \left| w(\xi,\tau) - w_{\rm app}(\xi,\tau) \right|
\leq C k^{3/2} e^{-\alpha k},
\end{eqnarray}
for any $t$ in the time interval (\ref{time-interval}). The error
bound (\ref{aprox-solution-bound}) is exponentially small as $k \to \infty$.

On the other hand, we can consider
\begin{equation}
\label{aprox-solution-new}
u_{\infty}(x,t) = p(t) \left( 2x - w_{\infty}(p(t) x) \right), \quad p(t) =
\frac{4 k}{1 + 16 k t},
\end{equation}
where $w_{\infty}(\xi) = \tanh(\xi)$. Bound (\ref{bound-asymptotic-solution})
in Corollary \ref{corollary-1} imply that there are $\delta > 0$ and $C > 0$ such that
\begin{equation}
\label{aprox-solution-bound-new}
\sup_{x \in \mathbb{T}} \left| u_{\rm app}(x,t) - u_{\infty}(x,t) \right| \leq
\frac{C}{k^{\delta}},
\end{equation}
in the inertial range
\begin{equation}
\label{time-interval-inertial}
\frac{C_{\infty}(k)}{1 - C_{\infty}(k)} \leq 16 k t \leq \frac{C_0}{1 - C_0},
\end{equation}
where
$$
C_{\infty}(k) = \frac{4 (1+\delta) \log(k)}{3 k} < C_0.
$$
We note that $C_{\infty}(k) \to 0$ as $k \to \infty$.

In the inertial range (\ref{time-interval-inertial}),
we can compute the leading order approximation of $E(u)$ and $K(u)$ from
the values of $E(u_{\infty})$ and $K(u_{\infty})$.
Using the representation (\ref{aprox-solution-new}), we obtain
\begin{equation}
\label{asympt-intermediate-range}
K(u_{\infty}) = \frac{1}{6} p^2(t) + \mathcal{O}(p(t)), \quad
E(u_{\infty}) = \frac{2}{3} p^3(t) + \mathcal{O}(p^2(t)),
\quad \mbox{\rm as} \quad p(t) \to \infty,
\end{equation}
On the other hand, $R(u)$ is approximated from $R(u_{\infty})$ by
\begin{equation}
\label{asympt-intermediate-range-R}
R(u_{\infty}) = -8 p^4(t) + {\cal O}(p^3(t)) \quad \mbox{\rm as} \quad p(t) \to \infty.
\end{equation}
Because $R(u_{\infty}) < 0$, the maximum of $E(u)$ occurs at the time $t = T_*
\leq \frac{C_{\infty}(k)}{16 k(1 - C_{\infty}(k))}$.
Moreover, it is clear that $E(u_{\infty})$ and hence $E(u)$ is
decreasing for all times $t \geq T_*$.

It remains to prove that $T_* = \mathcal{O}(k^{-2} \log(k))$ as $k \to \infty$ or, in other words,
that there exists $\tilde{C}_{\infty}(k) = \mathcal{O}(\log(k)/k) < C_{\infty}(k)$ such that
$T_*$ occurs inside
$$
\frac{\tilde{C}_{\infty}(k)}{16 k(1 - \tilde{C}_{\infty}(k))} \leq T_* \leq \frac{C_{\infty}(k)}{16 k(1 - C_{\infty}(k))}.
$$
If this is the case, then the scaling law (\ref{arg-max}) of Theorem \ref{theorem-time}
follows from the error bounds (\ref{aprox-solution-bound}) and (\ref{aprox-solution-bound-new}),
the triangle inequality, as well as from the previous computations:
$p(t) = \mathcal{O}(k)$ as $k \to \infty$ and $k = \mathcal{O}(\mathcal{E}^{1/3})$ as $\mathcal{E} \to \infty$.

To show that $T_* = \mathcal{O}(k^{-2} \log(k))$ as $k \to \infty$,
we compute $R(u_{\rm app})$ by using the explicit representation
(\ref{aprox-solution}) with $w_{\rm app}(\xi,\tau) =
\tanh(\xi) + \tilde{w}(\xi,\tau)$, where $\tilde{w}$ is given
by (\ref{approx-solution-w-2}).
Asymptotic computations yield
\begin{eqnarray*}
R(u_{\rm app}) & = & 2 p^5(t) \int_0^{p(t)/2} \left( w_{\rm apr, \xi}^3 - w_{\rm apr, \xi \xi}^2\right) d \xi
- 12 p^4(t) \int_0^{p(t)/2} w_{\rm apr, \xi}^2 d \xi + \mathcal{O}(p^3(t)) \\
& = & 2 N p^5(t) e^{-3\tau/4} - 8 p^4(t) + \mathcal{O}\left(e^{-3\tau/4} p^4(t), p^3(t)\right)
\end{eqnarray*}
as $e^{-3\tau/4} \to 0$ and $p(t) \to \infty$,
where $N$ is given by
\begin{eqnarray*}
N & = & \int_0^{\infty} \cosh(\xi/2)
\left( -28 {\rm sech}^3(\xi) + 139 {\rm sech}^5(\xi) - 120 {\rm sech}^7(\xi) \right) d \xi \\
& \phantom{t} & + \int_0^{\infty} \sinh(\xi/2) \sinh(\xi)
\left( 26 {\rm sech}^4(\xi) - 30 {\rm sech}^6(\xi) \right) d \xi.
\end{eqnarray*}
Numerical approximation of the integral shows that $N \approx 5.5189 > 0$.
(The positivity of $N$ also follows from the fact that $R$ is positive
for the approximate solution $w_{\rm app}$ shown on the right panel of Figure \ref{figure-2}.)
Therefore, $R(\tilde{u}_{\rm app}) = 0$ at the time $t = T_*$
(corresponding to $\tau = \tau_*$ by the transformation in (\ref{aprox-solution})),
when $p(T_*) = \mathcal{O}(e^{3\tau_*/4})$. Since $p(T_*) = \mathcal{O}(k)$ as $k \to \infty$
everywhere in (\ref{time-interval}), we have
$\tau_* = \mathcal{O}(\log(k)) \ll C_0 k$ or
$T_* = \mathcal{O}(k^{-2} \log(k))$ as $k \to \infty$.
The proof of Theorem \ref{theorem-time} for $l = k$ is now complete.

\section{Proof of Theorem \ref{theorem-time}: Case $l = \mathcal{O}(\log(k))$ as $k \to \infty$.}

We shall now consider the initial data (\ref{initial-data}) and (\ref{particular-function}).
By Lemma \ref{lemma-exp-asymptotics}, we fix $\Delta > 0$ and set
\begin{equation}
\label{choice-l}
l := (1 + \Delta) \log(k).
\end{equation}
This choice represents a more general case $l = \mathcal{O}(\log(k))$ as $k \to \infty$.

By Lemma \ref{lemma-solution-Burgers}, a solution of the Burgers equation (\ref{Burgers})
is written in the form
\begin{equation}
u(x,t) = p(t) \left( 2x - w(p(t) x,\tau) \right), \quad p(t) =
\frac{4 k}{1 + 16 k t}, \quad \tau = \frac{16 k^2 t}{1 + 16 k t},
\end{equation}
where $w$ solves the rescaled Burgers equation (\ref{Burgers-rescaled-first})
with the initial data
\begin{equation}
w_0(\xi) = \lambda \tanh(\xi/a) \quad \Rightarrow \quad
\psi_0(\xi) = \cosh^{\lambda a}(\xi/a),
\end{equation}
where $a = \frac{4k}{l}$ and $\lambda = \coth(l/2)$. The approximate solution
of the rescaled Burgers equation (\ref{Burgers-rescaled-first})
is given by Lemma \ref{lemma-Burgers-solution}.
It can be written in the form (\ref{approx-solution-w-1}) with
\begin{equation}
\psi_{\rm app} = \frac{1}{2^{a-1}} e^{\tau} \cosh(\xi)
\left[ 1 + \frac{e^{\xi} \psi_+(\xi,\tau)
+ e^{-\xi} \psi_-(\xi,\tau)}{\sqrt{\pi} (e^{\xi} + e^{-\xi})} \right],
\end{equation}
where $\psi_{\pm}$ are defined by (\ref{expression-psi-plus})--(\ref{expression-psi-minus}). Note
that $\psi_{{\rm app},0}(\xi) = \cosh^a(\xi/a)$. Assumption (\ref{positivity-slope})
is satisfied by the monotonicity of the transition from $w_{{\rm app},0}(\xi) = \tanh(\xi/a)$
to $w_{\infty}(\xi) = \tanh(\xi)$. Assumption (\ref{approx-solution-w-3}) follows from
bound (\ref{main-bound-line-infty-boundary}) with
\begin{equation}
\label{C-k-algebraic}
\exists C > 0 : \quad C_k = \frac{C k}{a^{1+\delta}} = \frac{C \log^{1+\delta}(k)}{k^{\delta}}, \quad
\tau \in [0,C_0 k],
\end{equation}
as long as
\begin{equation}
\label{C-k-constraint}
2(k-\tau) \geq \frac{1}{2} (1 + \delta)^2 a \log(a), \quad \tau \in [0,C_0 k].
\end{equation}
Constant $C_k$ is algebraically small as $k \to \infty$, whereas
the constraint (\ref{C-k-constraint}) yields an upper bound on $C_0$,
\begin{equation}
\label{C-0-constraint}
C_0 \leq C_0^* := 1 - \frac{1}{4 k} (1 + \delta)^2 a \log(a) = 1 - \frac{(1+\delta)^2}{(1 + \Delta)}
\left[ 1 - \frac{\log(l/4)}{\log(k)} \right].
\end{equation}
To ensure that $C_0^* > 0$, we require $(1 + \Delta) > (1 + \delta)^2$.

The initial condition,
$$
\Psi_0(\xi) = \cosh^{a(\lambda - 1)}(\xi/a) - 1 = \exp\left(\frac{2 a e^{-l}}{1 - e^{-l}}
\log\left[\cosh\left(\frac{\xi}{a}\right)\right]\right) - 1,
$$
implies that for sufficiently large $k \gg 1$, there is $C > 0$ such that
$$
\sup_{|\xi| \leq 2 k} |\Psi_0(\xi)| \leq C k e^{-l} = \frac{C}{k^{\Delta}}, \quad
\| \Psi_0 \|_{L^2_{k,0}} \leq \frac{C}{k^{\Delta-1/2}}, \quad
\| \Psi_0' \|_{L^2_{k,0}} + \| \Psi_0'' \|_{L^2_{k,0}} \leq \frac{C}{k^{\Delta+1/2}},
$$
hence the assumption (\ref{initial-estimate-small}) is satisfied with
$D_k = \frac{C}{k^{\Delta-1/2}}$ for some $C > 0$. Constant
$D_k$ is algebraically small as $k \to \infty$ provided $\Delta > \frac{1}{2}$.
Lemma \ref{lemma-new-approximation} yields
\begin{equation}
\label{bound-w-main-new-yet}
\| w - w_{\rm app} \|^2_{L^2_{k,\tau}} + \| \partial_{\xi}(w - w_{\rm app}) \|^2_{L^2_{k,\tau}} \leq
C \left( \frac{1}{k^{2 \Delta - 1}} + \frac{\log^{1+\delta}(k)}{k^{\delta}}\right),
\quad \tau \in [0,C_0 k].
\end{equation}

Applying Lemma \ref{lemma-solution-Burgers}, we write an approximate
solution of the Burgers equation (\ref{Burgers}) with initial data
(\ref{initial-data}) and (\ref{particular-function}) in the form,
\begin{equation}
\label{aprox-solution-last}
u_{\rm app}(x,t) = p(t) \left( 2x - w_{\rm app}(p(t) x,\tau) \right), \quad p(t) =
\frac{4 k}{1 + 16 k t}, \quad \tau = \frac{16 k^2 t}{1 + 16 k t},
\end{equation}
on the time interval
\begin{equation}
\label{time-interval-last}
0 \leq 16 k t \leq \frac{C_0}{1 - C_0}.
\end{equation}
Because $|p(t)| = \mathcal{O}(k)$
as $k \to \infty$ for any $t$ in the time interval (\ref{time-interval-last}),
bound (\ref{bound-w-main-new-yet}) and Sobolev embedding of $H^1$ to $L^{\infty}$
imply that there is constant $C > 0$ such that
\begin{eqnarray}
\nonumber
\sup_{x \in \mathbb{T}} \left| u(x,t) - u_{\rm app}(x,t) \right| & \leq &
C k \sup_{|\xi| \leq 2 (k - \tau)} \left| w(\xi,\tau) - w_{\rm app}(\xi,\tau) \right| \\
\label{aprox-solution-bound-last}
& \leq & C \left( \frac{1}{k^{\Delta - 3/2}} + \frac{\log^{(1+\delta)/2}(k)}{k^{\delta/2-1}}\right),
\end{eqnarray}
for any $t$ in the time interval (\ref{time-interval-last}). The error
bound (\ref{aprox-solution-bound-last}) is algebraically small as $k \to \infty$
provided $\Delta > \frac{3}{2}$ and $\delta > 2$, instead of our previous constraint $\Delta > \frac{1}{2}$.

On the other hand, we can consider
\begin{equation}
\label{aprox-solution-last-new}
u_{\infty}(x,t) = p(t) \left( 2x - w_{\infty}(p(t) x) \right), \quad p(t) =
\frac{4 k}{1 + 16 k t},
\end{equation}
where $w_{\infty}(\xi) = \tanh(\xi)$. Bounds (\ref{bound-asymptotic-solution-other-case})
and (\ref{definition-T-star}) in Corollary \ref{corollary-2} imply that
there is a positive constant $C$ such that
\begin{equation}
\label{aprox-solution-bound-last-new}
\sup_{x \in \mathbb{T}} \left| u_{\rm app}(x,t) - u_{\infty}(x,t) \right| \leq
\frac{C k}{a^{3 \delta + \delta^2+ 1/2} \log^{1/2}(a)} \leq
\frac{C \log^{3 \delta + \delta^2}(k)}{k^{3 \delta + \delta^2 - 1/2}},
\end{equation}
in the inertial range
\begin{equation}
\label{time-interval-inertial-last}
\frac{C_{\infty}}{1 - C_{\infty}} \leq 16 k t \leq \frac{C_0}{1 - C_0},
\end{equation}
where
$$
C_{\infty} = \frac{2 (1+\delta)^2}{(1+\Delta)} < C_0.
$$
Because of constraint (\ref{C-0-constraint}) on $C_0$, we realize that
$(1 + \Delta) > 3(1 + \delta)^2$ instead of our previous constraint $(1 + \Delta) > (1 + \delta)^2$.

In the inertial range (\ref{time-interval-inertial-last}),
the computations of $R(u_{\infty})$, $E(u_{\infty})$ and $K(u_{\infty})$
are identical to those in Section 6 and yield expressions (\ref{asympt-intermediate-range})
and (\ref{asympt-intermediate-range-R}).
In particular, $E(u_{\infty}) = \mathcal{O}(k^3)$ and $K(u_{\infty}) = \mathcal{O}(k^2)$
as $k \to \infty$. Modifications of the previous argument show
that the maximum of $E(u)$ at the time $t = T_*$
occurs for $T_* = \mathcal{O}(k^{-1})$ as $k \to \infty$.
By Lemma \ref{lemma-exp-asymptotics}, we have $k = \mathcal{O}(\mathcal{E}^{1/2} \log^{-1/2}(\mathcal{E}))$.
Substituting this into the previous expressions yields the scaling law
(\ref{arg-max-theorem}) of Theorem \ref{theorem-time} for $l = (1 + \Delta) \log(k)$.
The constraints on $\Delta$ and $\delta$ are consistent if
$$
\Delta > \frac{3}{2},  \quad \delta > 2, \quad \mbox{\rm and} \quad
(1 + \Delta) > 3(1 + \delta)^2,
$$
which can be satisfied, for instance, by the choice $\Delta = 48$ and $\delta = 3$.

\end{document}